\documentclass[a4paper]{amsart}
\usepackage{amsmath,amssymb,enumerate}
\usepackage{tikz}
\usepackage{hyperref}

\title[Connected Choice and the Brouwer Fixed Point Theorem]
{Connected Choice and\\ the Brouwer Fixed Point Theorem}

\author{Vasco Brattka}
\author{St\'{e}phane Le Roux}
\author{Joseph S.\ Miller}
\author{Arno Pauly}

\thanks{This research was supported by a Marie Curie International Research
Staff Exchange Scheme Fellowship within the 7th European Community Framework Programme and
by the National Research Foundation of South Africa.
Earlier versions of results from this article were presented at the conference {\em Computability in Europe} (CiE) in 2012 and 2016 \cite{BLP12a,BLRMP16}.\\
2010 Mathematics Subject Classification: 03D30, 03D78, 03F60, 37C25, computable analysis, Weihrauch lattice, Brouwer Fixed Point Theorem}

\address{Faculty of Computer Science, Universit\"at der Bundeswehr M\"unchen, Germany and Department of Mathematics and Applied Mathematics, University of Cape Town, South Africa} 
\address{Laboratoire Sp\'ecification et V\'erification, CNRS \& \'Ecole Normale Sup\'erieure Paris-Saclay, Cachan, France}
\address{Department of Mathematics, University of Wisconsin--Madison, USA}
\address{Department of Computer Science, Swansea University, United Kingdom}

\email{Vasco.Brattka@cca-net.de}
\email{leroux@lsv.fr}
\email{jmiller@math.wisc.edu}
\email{a.m.pauly@swansea.ac.uk}

\def\AA{{\mathcal A}}

\def\CC{{\mathcal C}}

\def\NN{{\mathcal N}}

\def\SS{{\mathcal S}}

\def\UU{{\mathcal U}}

\def\IN{{\mathbb{N}}}
\def\IQ{{\mathbb{Q}}}
\def\IR{{\mathbb{R}}}

\def\IZ{{\mathbb{Z}}}

\def\TO{\Longrightarrow}
\def\In{\subseteq}

\def\into{\hookrightarrow}

\def\prefix{\sqsubseteq}

\def\mto{\rightrightarrows}

\def\id{{\rm id}}

\def\dom{{\rm dom}}
\def\range{{\rm range}}

\def\diam{{\rm diam}}

\def\ind{{\rm ind}}

\def\Ls{{\rm Ls}}

\def\Cantor{{\{0,1\}^\IN}}
\def\Baire{{\IN^\IN}}

\def\cc{\mathrm c}

\def\LLPO{\text{\rm\sffamily LLPO}}
\def\WKL{\text{\rm\sffamily WKL}}

\def\IVT{\text{\rm\sffamily IVT}}

\def\BFT{\mbox{\rm\sffamily BFT}}

\def\C{\mbox{\rm\sffamily C}}
\def\ConC{\mbox{\rm\sffamily CC}}

\def\LLPO{\mbox{\rm\sffamily LLPO}}

\def\L{\text{\rm\sffamily L}}

\def\Con{\text{\rm\sffamily Con}}
\def\Fix{\text{\rm\sffamily Fix}}
\def\J{\text{\rm\sffamily J}}

\def\PWCC{\text{\rm\sffamily PWCC}}
\def\XC{\text{\rm\sffamily XC}}

\def\CQ{\mathrm C\IQ}

\def\leqW{\mathop{\leq_{\mathrm{W}}}}
\def\equivW{\mathop{\equiv_{\mathrm{W}}}}
\def\nequivW{\mathop{\not\equiv_{\mathrm{W}}}}
\def\leqSW{\mathop{\leq_{\mathrm{sW}}}}

\def\equivSW{\mathop{\equiv_{\mathrm{sW}}}}

\def\nleqW{\mathop{\not\leq_{\mathrm{W}}}}

\def\lW{\mathop{<_{\mathrm{W}}}}
\def\lSW{\mathop{<_{\mathrm{sW}}}}

\def\bigtimes{\mathop{\mathsf{X}}}

\newcommand{\W}{\mathrm{W}}

\date{\today}

\newtheorem{theorem}{Theorem}[section]
\newtheorem{proposition}[theorem]{Proposition}
\newtheorem{lemma}[theorem]{Lemma}
\newtheorem{fact}[theorem]{Fact}

\newtheorem{corollary}[theorem]{Corollary}
\newtheorem{localclaim}{Claim}[theorem]

\theoremstyle{definition}
\newtheorem{definition}[theorem]{Definition}
\newtheorem{localdef}[localclaim]{Definition}

\newtheorem{question}[theorem]{Question}

\begin{document}

\maketitle

\begin{abstract}
We study the computational content of the Brouwer Fixed Point Theorem in the Weihrauch lattice.
Connected choice is the operation that finds a point in a non-empty
connected closed set given by negative information. 
One of our main results is that for any fixed dimension the Brouwer Fixed Point Theorem of that dimension 
is computably equivalent to connected choice of the Euclidean unit cube of the same dimension. 
Another main result is that connected choice is complete
for dimension greater than or equal to two in the sense that it is computably equivalent
to Weak K\H{o}nig's Lemma. 
While we can present two independent proofs for dimension three and upwards that are either based on a simple geometric construction
or a combinatorial argument, 
the proof for dimension two is based on a more involved inverse limit construction.
The connected choice operation in dimension one is known to be equivalent to the Intermediate Value Theorem;
we prove that this problem is not idempotent in contrast to the case of dimension two and upwards.
We also prove that Lipschitz continuity with Lipschitz constants strictly larger than one does not
simplify finding fixed points. 
Finally, we prove that finding a connectedness component of a closed subset
of the Euclidean unit cube of any dimension greater or equal to one is equivalent to
Weak K\H{o}nig's Lemma.  
In order to describe these results, we introduce a representation
of closed subsets of the unit cube by trees of rational complexes.\\[0.2cm]
{\bf Keywords:} Computable analysis, Weihrauch lattice, reverse mathematics, 
choice principles, connected sets, fixed point theorems.
\end{abstract}

\tableofcontents
\pagebreak

\section{Introduction}
\label{sec:introduction}

In this paper, we continue with the programme to classify the computational content of mathematical
theorems in the Weihrauch lattice (see \cite{GM09,BG11,BG11a,Pau09,Pau10,BGM12,HRW12}).
This lattice is induced by Weihrauch reducibility, which is a reducibility for multi-valued partial functions 
$f:\In X\mto Y$ on represented spaces $X,Y$. 
Intuitively, $f\leqW g$ reflects the fact that $f$ can be realized
with a single application of $g$ as an oracle. Hence, if two multi-valued functions are equivalent
in the sense that they are mutually reducible to each other, then they are equivalent as 
computational resources, as far as computability is concerned.

Many theorems in mathematics are actually of the logical form
\[(\forall x\in X)(\exists y\in Y)\;P(x,y)\]
and such theorems can straightforwardly be represented by a multi-valued function $f:X\mto Y$
with $f(x):=\{y\in Y:P(x,y)\}$ (sometimes partial $f$ are needed, where the domain captures
additional requirements that the input $x$ has to satisfy).
In some sense, the multi-valued operation $f$ directly reflects the computational task of the 
theorem to find some suitable $y$ for any $x$.
Hence, in a very natural way the classification of a theorem can be achieved via a classification of
the corresponding multi-valued function that represents the theorem.

Theorems that have been compared and classified in this sense include Weak K\H{o}nig's Lemma $\WKL$,
the Hahn-Banach Theorem~\cite{GM09}, the Baire Category Theorem~\cite{BHK18}, Banach's Inverse Mapping Theorem,
the Open Mapping Theorem, the Uniform Boundedness Theorem, the Intermediate Value
Theorem~\cite{BG11a}, the Bolzano-Weierstra\ss{} Theorem~\cite{BGM12}, 
Nash Equilibria~\cite{Pau10}, the Radon-Nikodym Theorem~\cite{HRW12} and the Vitali Covering Theorem~\cite{BGHP17}. 
In this paper, we to classify the Brouwer Fixed Point Theorem.

\begin{theorem}[Brouwer Fixed Point Theorem 1911]
Every continuous function $f:[0,1]^n\to[0,1]^n$ has a fixed point $x\in[0,1]^n$, i.e., a point such
that $f(x)=x$.
\end{theorem}

This theorem was first proved by Hadamard in 1910 and later by Brouwer~\cite{Bro11},
after whom it is named the Brouwer Fixed Point Theorem.
Brouwer is known as one of the founders of intuitionism, which is one of the well-studied
varieties of constructive mathematics and ironically, the theorem that he is best known for
does not admit any constructive proof.\footnote{However, as noticed already by Brouwer himself, 
the theorem admits an approximative constructive version~\cite{Bro52}.}
This fact has been confirmed in many different ways, most relevant for us is the 
counterexample in Russian constructive analysis by Orevkov~\cite{Ore63}, which
was transferred into computable analysis by Baigger~\cite{Bai85}.
Baigger's counterexample shows that from dimension two upwards (i.e., $n\geq2$) there are computable
functions $f:[0,1]^n\to[0,1]^n$ without computable fixed point $x$.
Baigger's proof actually proceeds by encoding a Kleene tree (implicitly via a pair of computably inseparable sets) 
into a suitable computable function $f$ and hence it can be seen as a reduction
of Weak K\H{o}nig's Lemma to the Brouwer Fixed Point Theorem.\footnote{See \cite{Pot08} for a discussion of these counterexamples.}
The essential geometrical content of this construction is that the map
\[A\mapsto (A\times[0,1])\cup([0,1]\times A)\]
maps {\em arbitrary} non-empty closed sets $A\In[0,1]$ to {\em connected} non-empty closed
subsets of $[0,1]^2$ such that any pair in the resulting set has at least one component that is in $A$. 

Constructions similar to those used for the above counterexamples 
have been utilized in order to prove that the Brouwer Fixed Point Theorem
is equivalent to Weak K\H{o}nig's Lemma in reverse mathematics \cite{Sim99,ST90,Koh05b}
and to analyze computability properties of fixable sets \cite{Mil02a},
but a careful analysis of these reductions reveals that none of them can be 
straightforwardly transferred into a {\em uniform} reduction in the sense that we are seeking here.
The problem is that there is no uniform way to select a component $x_i$ of a pair $(x_1,x_2)$ such that  $x_i\in A$, 
given that at least one of the components has this property.
The results cited above essentially characterize the complexity of the fixed points themselves, whereas we
want to characterize the complexity of finding a fixed point, given the function.
This requires full uniformity.

In the Weihrauch lattice, the Brouwer Fixed Point Theorem of dimension $n$ is represented by the 
multi-valued function $\BFT_n:\CC([0,1]^n,[0,1]^n)\mto[0,1]^n$ that maps any 
continuous function $f:[0,1]^n\to[0,1]^n$ to the set of its fixed points $\BFT_n(f)\In[0,1]^n$.
The question now is where $\BFT_n$ is located in the Weihrauch lattice? 
It easily follows from a meta theorem presented in \cite{BG11a} that the Brouwer Fixed Point Theorem $\BFT_n$ is
reducible to Weak K\H{o}nig's Lemma $\WKL$ for any dimension $n$, i.e., $\BFT_n\leqW\WKL$.
However, for which dimensions $n$ do we also obtain the inverse reduction?
Clearly not for $n=0$, since $\BFT_0$ is computable, and clearly not for $n=1$, since 
$\BFT_1$ is equivalent to the Intermediate Value Theorem $\IVT$ and hence not equivalent
to $\WKL$, as proved in \cite{BG11a}.\footnote{We mention that 
in Bishop style constructive reverse mathematics the Intermediate Value Theorem is equivalent to Weak K\H{o}nig's Lemma~\cite{Ish06}, as parallelization is freely available in this framework.} 

In order to approach this question for a general dimension $n$, 
we introduce a choice principle $\ConC_n$ that we call {\em connected choice}
and which is just the closed choice operation restricted to connected subsets. That is, in the sense discussed above $\ConC_n$
is the multi-valued function that  
represents the following mathematical statement: every non-empty connected closed set $A\In[0,1]^n$ has a point
$x\in[0,1]^n$. Since closed sets are represented by negative information (i.e., by an enumeration of open balls
that exhaust the complement), the computational task of $\ConC_n$ consists in finding a point in a closed set $A\In[0,1]^n$
that is promised to be non-empty and connected and that is given by negative information. 

One of our main results, proved in Section~\ref{sec:BFT}, is that the Brouwer Fixed Point Theorem is equivalent to connected choice for 
each fixed dimension $n$, i.e., 
\[\BFT_n\equivW\ConC_n.\] 
This result allows us to study the Brouwer Fixed
Point Theorem in terms of the operation $\ConC_n$ that is easier to handle since
it involves neither function spaces nor fixed points. This is also another instance of the observation
that several important theorems are equivalent to certain choice principles (see \cite{BG11a}) and
many important classes of computable functions can be calibrated in terms of choice (see \cite{BBP12}).
For instance, closed choice on Cantor space $\C_{\{0,1\}^\IN}$ and on the unit cube $\C_{[0,1]^n}$ are both
easily seen to be equivalent to Weak K\H{o}nig's Lemma $\WKL$, i.e., $\WKL\equivW\C_{\{0,1\}^\IN}\equivW\C_{[0,1]^n}$ 
for any $n\geq1$. 
Studying the Brouwer Fixed  Point Theorem in form of $\ConC_n$ now amounts to comparing
$\C_{[0,1]^n}$ with its restriction $\ConC_n$.

Our second main result, proved in Sections~\ref{sec:dimension} and \ref{sec:dim2},
is that from dimension two upwards connected choice is equivalent to
Weak K\H{o}nig's Lemma, i.e., $\ConC_n\equivW\C_{[0,1]}$ for $n\geq2$.
In Section~\ref{sec:dimension}, we present a proof for dimension $n\geq3$
that is based on the geometrical construction
\[A\mapsto(A\times[0,1]\times\{0\})\cup(A\times A\times[0,1])\cup([0,1]\times A\times\{1\})\]
that maps an {\em arbitrary} non-empty closed set $A\In[0,1]$ to a {\em pathwise connected} non-empty closed subset of $[0,1]^3$ 
that has the property that from any of its points we can compute a point of the original set $A$ in a uniform sense.
This construction seems to require at least dimension three in a crucial sense.
The same is true for an alternative combinatorial proof of the same result that we provide.
The proof for dimension 2 is presented in Section~\ref{sec:dim2} and is based on 
a more involved inverse limit construction and hence on an entirely different idea.
It only yields a connected (not necessarily pathwise connected) set in general. 
We are left with the open question whether pathwise connected choice of dimension two is 
equivalent to connected choice of dimension two. 

In Section~\ref{sec:Lipschitz}, we show that Lipschitz continuity with a Lipschitz constant $L>1$ does
not simplify finding fixed points. Using results of Neumann~\cite{Neu15}, we obtain a trichotomy
of the problem of finding fixed points for Lipschitz continuous functions that depends on 
whether the Lipschitz constant $L$ satisfies $L<1$, $L=1$ or $L>1$.

In order to prove some of our results, we use a representation of closed sets by 
trees of so-called {\em rational complexes}, which we introduce
in Section~\ref{sec:closed-sets-trees}.
It can be seen as a generalization of the well-known representation of
co-c.e.\ closed subsets of Cantor space $\{0,1\}^\IN$ by binary trees.
As a side result we prove that finding a connectedness component
of a closed set  for any fixed dimension from one upwards is 
equivalent to Weak K\H{o}nig's Lemma.
This yields conclusions along the line of earlier studies of 
connected components in \cite{LZ08a}.

Finally, we provide a so-called Displacement Principle in Section~\ref{sec:displacement}
that helps us in Section~\ref{sec:idempotency} to show  that $\ConC_1$ is neither idempotent
nor a cylinder.

In the following Section~\ref{sec:Weihrauch-lattice}, we start
with a short summary of relevant definitions and results regarding the
Weihrauch lattice.

\section{The Weihrauch Lattice}
\label{sec:Weihrauch-lattice}

In this section, we briefly recall some basic results and definitions regarding
the Weihrauch lattice. The original definition of Weihrauch reducibility is due to Weihrauch
and has been studied for many years (see \cite{Ste89,Wei92a,Wei92c,Her96,Bra99,Bra05}).
Only recently it has been noticed that a certain variant of this reducibility yields
a lattice that is very suitable for the classification of mathematical theorems
(see  \cite{GM09,Pau09,Pau10,BG11,BG11a,BBP12,BGM12}). The basic reference for all notions
from computable analysis is \cite{Wei00}, and a survey on Weihrauch complexity can be found in \cite{BGP17}. 
The Weihrauch lattice is a lattice of multi-valued functions on represented
spaces. A {\em representation} $\delta$ of a set $X$ is just a surjective partial
map $\delta:\In\IN^\IN\to X$. In this situation we call $(X,\delta)$ a {\em represented space}.
In general we use the symbol ``$\In$'' in order to indicate that a function is potentially partial.
Using represented spaces we can define the concept of a realizer. 
For $f:\In X\mto Y$ and $g:\In Y\mto Z$ we write $g\circ f$ or $gf$ for the {\em composition}
defined by $(g\circ f)(x):=\{z\in Z:(\exists y\in f(x))\;z\in g(y)\}$ and $\dom(g\circ f):=\{x\in \dom(f):f(x)\In\dom(g)\}$.

\begin{definition}[Realizer]
Let $f : \In (X, \delta_X) \mto (Y, \delta_Y)$ be a multi-valued function on represented spaces.
A function $F:\In\IN^\IN\to\IN^\IN$ is called a {\em realizer} of $f$, in symbols $F\vdash f$, if
$\delta_YF(p)\in f\delta_X(p)$ for all $p\in\dom(f\delta_X)$.
\end{definition}

Realizers allow us to transfer the notions of computability
and continuity and other notions available for Baire space to any represented space;
a function between represented spaces will be called {\em computable}, if it has a computable realizer, etc.
Now we can define Weihrauch reducibility.
By $\langle\;,\;\rangle:\Baire\times\Baire\to\Baire$ we denote the standard pairing function, defined by
$\langle p,q\rangle(2n):=p(n)$ and $\langle p,q\rangle(2n+1):=q(n)$ for all $p,q\in\Baire$ and $n\in\IN$.

\begin{definition}[Weihrauch reducibility]
Let $f,g$ be multi-valued functions on represented spaces. 
Then $f$ is said to be {\em Weihrauch reducible} to $g$, in symbols $f\leqW g$, if there are computable
functions $K,H:\In\IN^\IN\to\IN^\IN$ such that $K\langle \id, GH \rangle \vdash f$ for all $G \vdash g$.
Moreover, $f$ is said to be {\em strongly Weihrauch reducible} to $g$, in symbols $f\leqSW g$,
if there are computable functions $K,H$ such that $KGH\vdash f$ for all $G\vdash g$.
\end{definition}

The difference between ordinary and strong Weihrauch reducibility is that the ``output modifier'' $K$ has
direct access to the original input in case of ordinary Weihrauch reducibility, but not in case of strong Weihrauch reducibility. 
We note that the relations $\leqW$, $\leqSW$ and $\vdash$ implicitly refer to the underlying representations, which
we will only mention explicitly if necessary. It is known that these relations only depend on the underlying equivalence
classes of representations, but not on the specific representatives (see Lemma~2.11 in \cite{BG11}).
The relations $\leqW$ and $\leqSW$ are reflexive and transitive, thus they induce corresponding partial orders on the sets of 
their equivalence classes (which we refer to as {\em Weihrauch degrees} or {\em strong Weihrauch degrees}, respectively).
These partial orders will be denoted by $\leqW$ and $\leqSW$ as well. In this way one obtains a distributive bounded lattice for $\leqW$
which we call the {\em Weihrauch lattice} (for details see \cite{Pau09} and \cite{BG11}).
We use $\equivW$ and $\equivSW$ to denote the respective equivalences regarding $\leqW$ and $\leqSW$, 
and by $\lW$ and $\lSW$ we denote strict reducibility.

The Weihrauch lattice is equipped with a number of useful algebraic operations that we summarize in the next definition.
We use $X\times Y$ to denote the ordinary set-theoretic {\em product}, $X\sqcup Y:=(\{0\}\times X)\cup(\{1\}\times Y)$ in order
to denote {\em disjoint sums} or {\em coproducts}, by $\bigsqcup_{i=0}^\infty X_i:=\bigcup_{i=0}^\infty(\{i\}\times X_i)$ we denote the 
{\em infinite coproduct}. By $X^i$ we denote the $i$--fold product of a set $X$ with itself, where $X^0=\{()\}$ is some canonical computable singleton.
By $X^*:=\bigsqcup_{i=0}^\infty X^i$ we denote the set of all {\em finite sequences over $X$}
and by $X^\IN$ the set of all {\em infinite sequences over $X$}. 
All these constructions have parallel canonical constructions on representations and the corresponding representations
are denoted by $[\delta_X,\delta_Y]$ for the product of $(X,\delta_X)$ and $(Y,\delta_Y)$, $\delta_X\sqcup\delta_Y$
for the coproduct and $\delta^*_X$ for the representation of $X^*$ and $\delta_X^\IN$ for the representation
of $X^\IN$ (see \cite{BG11,Pau09,BBP12} for details). We will always assume that these canonical representations
are used, if not mentioned otherwise. 

\begin{definition}[Algebraic operations]
\label{def:algebraic-operations}
Let $f:\In X\mto Y$ and $g:\In Z\mto W$ be multi-valued functions on represented spaces. Then we define
the following operations:
\begin{enumerate}
\itemsep 0.2cm
\item $f\times g:\In X\times Z\mto Y\times W, (f\times g)(x,z):=f(x)\times g(z)$ \hfill (product)
\item $f\sqcap g:\In X\times Z\mto Y\sqcup W, (f\sqcap g)(x,z):=f(x)\sqcup g(z)$ \hfill (sum)
\item $f\sqcup g:\In X\sqcup Z\mto Y\sqcup W$, with $(f\sqcup g)(0,x):=\{0\}\times f(x)$ and\\
        $(f\sqcup g)(1,z):=\{1\}\times g(z)$ \hfill (coproduct)
\item $f^*:\In X^*\mto Y^*,f(i,x):=\{i\}\times f^i(x)$ \hfill (finite parallelization)
\item $\widehat{f}:\In X^\IN\mto Y^\IN,f(x_n):=\bigtimes_{i=0}^\infty f(x_i)$ \hfill (parallelization)
\end{enumerate}
\end{definition}

In this definition and in general we denote by $f^i:\In X^i\mto Y^i$ the $i$--th fold product
of the multi-valued map $f$ with itself. For $f^0$ we assume that $X^0:=\{()\}$ is a canonical
singleton for each set $X$ and hence $f^0$ is just the constant operation on that set.
It is known that $f\sqcap g$ is the {\em infimum} of $f$ and $g$ with respect to strong as well as
ordinary Weihrauch reducibility (see \cite{BG11}, where this operation was denoted by $f\oplus g$).
Correspondingly, $f\sqcup g$ is known to be the {\em supremum} of $f$ and $g$ with respect to $\leqW$ (see \cite{Pau09}).
The two operations $f\mapsto\widehat{f}$ and $f\mapsto f^*$ are known to be {\em closure operators}
in the corresponding lattices, which means
$f \leqW \widehat{f}$ and $\widehat{f} \equivW\,\widehat{\!\!\widehat{f}}$, and $f \leqW g$ implies $\widehat{f} \leqW \widehat{g}$
and analogously for finite parallelization (see \cite{BG11,Pau09}).

We use some terminology related to these algebraic operations. 
We say that $f$ is a {\em a cylinder} if $f\equivSW\id\times f$ where $\id:\Baire\to\Baire$ always
denotes the identity on Baire space, if not mentioned otherwise. 
Cylinders $f$ have the property that $g\leqW f$ is equivalent to $g\leqSW f$ (see \cite{BG11}).
We say that $f$ is {\em idempotent} if $f\equivW f\times f$.
We say that a multi-valued function on represented spaces is {\em pointed}, if it has a computable
point in its domain. For pointed $f$ and $g$ we obtain $f\sqcup g\leqSW f\times g$. 
If $f\sqcup g$ is idempotent, then we also obtain the inverse reduction. 
The finite parallelization $f^*$ can also be considered as {\em idempotent closure} as 
$f\equivW f^*$ holds if and only if $f$ is idempotent and pointed. 
We call $f$ {\em parallelizable} if $f\equivW\widehat{f}$ and it is easy to see that $\widehat{f}$ is always idempotent.
The properties of pointedness and idempotency are both preserved under
equivalence and hence they can be considered as properties of the respective degrees.

A particularly useful multi-valued function in the Weihrauch lattice is closed choice (see \cite{GM09,BG11,BG11a,BBP12})
and it is known that many notions of computability can be calibrated using the right version of choice. 
We will focus on closed choice for computable metric spaces, which are separable metric spaces
such that the distance function is computable on the given dense subset.
We assume that computable metric spaces are represented via their Cauchy representation
(see \cite{Wei00} for details).

By $\AA_-(X)$ we denote the set of closed subsets of a metric space $X$, where the index ``$-$'' indicates
that we work with negative information. These are given by the representation $\psi_-:\IN^\IN\to\AA_-(X)$, defined by
$\psi_-(p):=X\setminus\bigcup_{i=0}^\infty B_{p(i)}$,
where $B_n$ is some standard enumeration of the open balls of $X$ with center in the dense subset and rational radius.
The computable points in $\AA_-(X)$ are called {\em co-c.e.\ closed sets}.
We now define closed choice for the case of computable metric spaces.

\begin{definition}[Closed Choice]
Let $X$ be a computable metric space. Then the {\em closed choice} operation 
of this space $X$ is defined by
$\C_X:\In\AA_-(X)\mto X,A\mapsto A$
with $\dom(\C_X):=\{A\in\AA_-(X):A\not=\emptyset\}$.
\end{definition}

Intuitively, $\C_X$ takes as input a non-empty closed set in negative description (i.e., given by $\psi_-$) 
and it produces an arbitrary point of this set as output.
Hence, $A\mapsto A$ means that the multi-valued map $\C_X$ maps
the input $A\in\AA_-(X)$ to the set $A\In X$ as a set of possible outputs.
We mention a couple of properties of closed choice for specific spaces.

The omniscience principle $\LLPO$ has turned out to be very useful and it is closely related to
closed choice. We recall the definition.

\begin{definition}[Omniscience principle]
We define $\LLPO:\In\IN^\IN\mto\IN$ by
\[j\in\LLPO(p)\iff(\forall n\in\IN)\;p(2n+j)=0\]
for all $j\in\{0,1\}$, where
$\dom(\LLPO):=\{p\in\IN^\IN:p(k)\not=0$ for at most one $k\}$.
\end{definition}

It is easy to see that $\C_{\{0,1\}}\equivSW\LLPO$.
We mention that closed choice can also be used to characterize the computational content of many
theorems. By $\WKL$ we denote the straightforward formalization of Weak K\H{o}nig's Lemma. 
Since we will not use $\WKL$ in any formal sense here, we refer the reader to \cite{GM09,BG11} for precise definitions.

\begin{fact}[Weak K\H{o}nig's Lemma]
\label{fact:WKL}
$\WKL\equivSW\C_{\Cantor}\equivSW\C_{[0,1]^n}\equivSW\C_{[0,1]^\IN}\equivSW\widehat{\LLPO}$ for all $n\geq1$.
\end{fact}

\section{Closed Sets and Trees of Rational Complexes}
\label{sec:closed-sets-trees}

In this section, we want to describe a representation of closed sets $A\In[0,1]^n$ that is useful for the study of connectedness. 
It is well-known that closed subsets of Cantor space can be characterized exactly as sets of infinite paths of trees (see for instance \cite{CR98}).
We describe a similar representation of closed subsets of the unit cube $[0,1]^n$ of the Euclidean space.
While in the case of Cantor space clopen balls are associated to each node of the tree, we now associate
finite complexes of rational balls to each node. While infinite paths lead to points of the closed set in case of Cantor space,
they now lead to connectedness components (which can be seen as a generalization, since the connectedness components in
Cantor space are singletons).

This representation of closed subsets $A\In[0,1]^n$ of the unit cube will enable us to analyze the relation between connected
choice and the Brouwer Fixed Point Theorem in the next section. In this section we will use this representation
in order to prove the result that finding a connectedness component of a closed set $A$ is computationally exactly as difficult
as Weak K\H{o}nig's Lemma. 

We first fix some topological terminology that we are going to use.
We work with the {\em maximum norm} $||\;||$ on $\IR^n$, defined by $||(x_1,...,x_n)||:=\max\{|x_i|:i=1,...,n\}$.
By $d(x,A):=\inf_{a\in A}||x-a||$ we denote the {\em distance} of $x\in\IR^n$ to $A\In\IR^n$.
By $d_A:\IR^n\to\IR$ we denote the corresponding {\em distance function} given by $d_A(x):=d(x,A)$.
%By $||f||:=\sup_{x\in[0,1]^n}||f(x)||$ we denote the {\em supremum norm} for continuous functions $f:[0,1]^n\to[0,1]^n$.
By $B(x,r):=\{y\in\IR^n:||x-y||<r\}$ we denote the {\em open ball} with center $x$ and radius $r$ and by
$B[x,r]:=\{y\in\IR^n:||x-y||\leq r\}$ the corresponding {\em closed ball}.
Since we are using the maximum norm, all these balls are open or closed cubes, respectively (if the radius is positive).
By $\partial A$ we denote the topological {\em boundary},
by $\overline{A}$ the {\em closure} and by $A^\circ$ the {\em interior} of a set $A$.
If the underlying space $X$ is clear from the context, then $A^{\rm c}:=X\setminus A$ denotes the {\em complement} of $A$.

We are now prepared to define rational complexes, which are just finite sets of rational closed balls whose union is connected
and that pairwise intersect at most on their boundary. 

\begin{definition}[Rational complex]
We call a set $R:=\{B[c_1,r_1],...,B[c_k,r_k]\}$ of finitely many
closed balls $B[c_i,r_i]$ with rational center $c_i\in\IQ^n$ and 
positive rational radius $r_i\in\IQ$ an {\em ($n$--dimensional) rational complex}
if $\bigcup R$ is connected and $B_1,B_2\in R$ with $B_1\not=B_2$ 
implies $B_1^\circ\cap B_2^\circ=\emptyset$.
We say that a rational complex is {\em non-empty}, if $\bigcup R\not=\emptyset$.
By $\CQ^n$ we denote the set of $n$--dimensional rational complexes.
\end{definition}

Each rational complex $R$ can be represented by a list of the corresponding
rational numbers $c_1,r_1,...,c_k,r_k$ and we implicitly assume in the following
that this representation is used for the set of rational complexes $\CQ^n$.

In order to organize the rational complexes that are used to approximate sets
it is suitable to use trees.
We recall that a {\em tree} is a set $T\In\IN^*$ that is closed under prefix,
i.e., $u\prefix v$ and $v\in T$ implies $u\in T$.
A function $b:\IN\to\IN$ is called a {\em bound} of a tree $T$ if $w\in T$ implies $w(i)\leq b(i)$ for all $i=0,...,|w|-1$,
where $|w|$ denotes the {\em length} of the word $w$. A tree is called {\em finitely branching}, if it has a bound.
A tree of rational complexes is understood to be a finitely branching tree $T$ (together with a bound) such that to each node
of the tree a rational complex is associated, with the property that these complexes are compactly included in each
other if we proceed along paths of the tree and they are disjoint on any particular level of the tree.
We write $A\Subset B$ for two sets $A,B\In\IR^n$ if the closure $\overline{A}$ of $A$ is included in the interior $B^\circ$ of $B$
and we say that $A$ is {\em compactly included} in $B$ in this case.

\begin{definition}[Tree of rational complexes]
We call $(T,f)$ a {\em tree of rational complexes} if
$T\In\IN^*$ is a finitely branching tree and $f:T\to\CQ^n$ is
a function such that for all $u,v\in T$ with $u\not=v$
\begin{enumerate}
\item $u\prefix v\TO\bigcup f(v)\Subset\bigcup f(u)$,
\item $|u|=|v|\TO\bigcup f(u)\cap\bigcup f(v)=\emptyset$.
\end{enumerate}
\end{definition}

In the following we assume that finitely branching trees $T$ are represented as a pair $(\chi_T,b)$, where
$\chi_T:\IN^*\to\{0,1\}$ is the characteristic function of $T$ and $b:\IN\to\IN$ is a bound of $T$.
Correspondingly, trees $(T,f)$ of rational complexes are then represented in a canonical way by $(\chi_T,b,f)$.
We now define which set $A\In[0,1]^n$ is represented by such a tree $(T,f)$ of rational complexes. 

\begin{definition}[Closed sets and trees of rational complexes]
We say that a closed set $A\In\IR^n$ is {\em represented} by a tree $(T,f)$ of $n$--dimensional rational complexes
if one obtains $A=\bigcap_{i=0}^\infty\bigcup_{w\in T\cap\IN^i}\bigcup f(w)$.
\end{definition}

It is clear that in this way any tree $(T,f)$ of rational complexes actually represents a compact set $A$.
This is because $\bigcup f(w)$ is compact for each $w\in T$ and since $T$ is finitely branching, the set $T\cap\IN^i$
is finite for each $i$, hence $\bigcup_{w\in T\cap\IN^i}\bigcup f(w)$ is compact and hence $A$ is compact too.
Vice versa, every compact set $A\In\IR^n$ can be represented by a tree of $n$--dimensional rational complexes.
For $[0,1]^n$ we prove the uniform result that even the map $(T,f)\mapsto A$ is computable and has a computable multi-valued 
right inverse. We assume that trees of rational complexes are represented as specified above and closed sets $A$
are represented as points in $\AA_-([0,1]^n)$. 
We recall that a {\em connectedness component} of a set $A$ is a connected subset of $A$ that is not included in any larger connected
subset of $A$. Any connectedness component of a subset $A$ is closed in $A$. If $A=\emptyset$, then the only connectedness
component is the empty set, otherwise connectedness components are always non-empty.

\begin{proposition}[Closed sets and complexes]
\label{prop:complexes}
Let $n\geq1$.
The map $(T,f)\mapsto A$ that maps every tree of $n$--dimensional rational complexes $(T,f)$ to the closed set $A\In[0,1]^n$ represented by it,
is computable and has a multi-valued computable right inverse.
An analogous result holds restricted to infinite trees of non-empty rational complexes and non-empty closed $A$.
\end{proposition}
\begin{proof}
It is clear that, given $(T,f)$ and a bound $b$ of $T$ we can actually compute $A\in\AA_-([0,1]^n)$.
Firstly, we can explicitly determine all finitely many $w\in T\cap\IN^i$ using the bound $b$
and compute $C_i:=\bigcup_{w\in T\cap\IN^i}\bigcup f(w)\in\AA_-([0,1]^n)$ for each $i$.
Since intersection of sequences of closed sets is computable on $\AA_-([0,1]^n)$, we can also compute $A:=\bigcap_{i=0}^\infty C_i$.

We note that if $(T,f)$ is an infinite tree of non-empty rational complexes then the $C_i$ form a decreasing chain of non-empty compact sets
and hence $A=\bigcap_{i=0}^\infty C_i$ is non-empty too by Cantor's Intersection Theorem.

For the other direction, let us assume that $A\In[0,1]^n$ is given as the complement of a union of rational open balls $B(c_i,r_i)$.
We use the larger cube $Q:=[-1,2]^n$ and we assume that $A=Q\cap(\bigcup_{i=0}^\infty B(c_i,r_i))^\cc$ with $B(c_i,r_i)\cap Q\not=\emptyset$ for all $i$.
Now we show how we can compute a tree $(T,f)$ of rational complexes together with a bound $b$ that represents $A$.
We proceed inductively over the length $i=|w|=0,1,2,...$ of words in the tree $T$.

We start with the empty node $\varepsilon\in T$ and we assign $f(\varepsilon)=\{Q\}$ to it.
Let us now assume that $T\cap\IN^i$ has been completely determined, $f(w)$ has been fixed for all $w\in T\cap\IN^i$
and $b(j)$ has been determined for all $j<i$. We now determine $T\cap\IN^{i+1}$, $f(w)$ for words $w\in T\cap\IN^{i+1}$,
and $b(i)$. The following is applied to each $w\in T\cap\IN^i$:
\begin{enumerate}
\item Firstly, we copy each rational complex $f(w)$ into $f(w0)$.
\item Then the points $B:=\{x:d(x,\partial\bigcup f(w))<2^{-i-1}\}$ which are close to the boundary are removed from $\bigcup f(w0)$.
         That is $f(w0)$ is refined such that the resulting union is the original one minus $B$ and all new balls in $f(w0)$ intersect at most on their boundaries.
        This guarantees $\bigcup f(w0)\Subset\bigcup f(w)$ (but it might destroy the property that $\bigcup f(w0)$ is connected).
\item In the next step $U:=\bigcup_{j=0}^i B(c_j,r_j-2^{-i})$ is removed from $\bigcup f(w0)$. This means that $f(w0)$ is refined
        such that the union is the original union minus $U$ and all new balls in $f(w0)$ intersect at most on their boundaries.
        This guarantees that the tree of rational complexes will eventually represent $A$ (we subtract $2^{-i}$ from the radius here in order to ensure
        that there is enough space for removing the boundary stripe $B$ in the next step (2) of the induction without removing anything of $A$).
\item Now $\bigcup f(w0)$ is not necessarily connected, but it has only finitely many connectedness components $C_0,...,C_k$
        that can all be explicitly determined as rational complexes. We copy these rational complexes into $f(w0),...,f(wk)$
        such that $\bigcup f(wj)=C_j$ for $j=0,...,k$ afterwards.
        Then all the $\bigcup f(wj)$ are pairwise disjoint and $\bigcup f(wj)\Subset\bigcup f(w)$ for all $j=0,...,k$.
        Should the only connectedness component $C_0$ be the empty set, then we stop the tree $T$ at this point and add no words $wj$ to it.
\end{enumerate}
After this procedure has been completed for all finitely many $w\in T\cap\IN^i$, we choose $b(i)$ as the maximal number
$k$ (of connectedness components) that occurred for any of these words $w$. It is clear that then $v(i)\leq b(i)$ for all $v\in T\cap\IN^{i+1}$. 
Moreover, we also have $\bigcup f(wj)\cap\bigcup f(vl)=\emptyset$ for all $w,v\in\IN^{i+1}$ with $v\not=w$ since 
$\bigcup f(w)\cap\bigcup f(v)=\emptyset$ and $\bigcup f(wj)\Subset\bigcup f(w)$ and $\bigcup f(vl)\Subset\bigcup f(v)$.

Altogether, $(T,f)$ as constructed here is a tree of rational complexes with bound $b$.
We still need to prove that the set $A_{(T,f)}$ represented by $(T,f)$ is actually $A$.
Let us denote by $A_i:=\bigcup_{w\in T\cap\IN^i}\bigcup f(w)$ the closed set represented by the union of all the complexes
of height $i$. In particular $A_{(T,f)}=\bigcap_{i\in\IN}A_i$. 
If $x\in Q\setminus A$, then there are some $i,j$ such that $x\in B(c_j,r_j-2^{-i})$ and hence $x$ is removed from all the complexes
of height $i$ of the tree in step (3) above. Hence $x\not\in A_i$, which implies $A_{(T,f)}\In A$.
Let, on the other hand, $x\in A$. Then clearly $x\in A_0=Q$ and has distance from $\partial A_0$ at least $1$.
By induction one can show that for each $i$ the distance $d(x,\partial A_i)$ is at least $2^{-i}$ and hence $x$ cannot 
be removed in step (2) (and also not in step (3) since only points outside $A$ are removed there).  
This induction shows that $x\in A_i$ for all $i$ and hence $x\in A_{(T,f)}$.
Altogether we have proved $A=A_{(T,f)}$.

We note that if $A$ is a non-empty set, then there is always at least one non-empty connectedness component $C_0$ in step (4) of the 
algorithm and the computed tree is automatically an infinite tree of non-empty rational complexes. If $A$ is the empty set,
then the computed tree is a finite tree of non-empty rational complexes.
\end{proof}

We note that this proof in particular shows that we can restrict the investigation in general to trees of non-empty rational complexes
(even if we want to include the empty closed set).
The previous result has a lot of interesting applications. For instance, if $A$ is represented by $(T,f)$, then the sets 
$A_i:=\bigcup_{w\in T\cap\IN^i}\bigcup f(w)$ of height $i$ used in the previous proof are of very special form.
They are finite unions of connected sets that are themselves finite unions of rational closed balls.
It is easy to see that for a co-c.e.\ closed $A$ the resulting sequence $(A_i)_{i\in\IN}$ is automatically
a computable sequence of {\em bi-computable} sets $A_i$, which means that the sequences $(d_{A_i})_{i\in\IN}$
and $(d_{A_i^\cc})_{i\in\IN}$ of characteristic functions are computable (see \cite{Her99b} for more information on bi-computable sets).
This is because the maps $R\mapsto d_{\bigcup R}$ and and $R\mapsto d_{(\bigcup R)^\cc}$ of type $\CQ^n\to\CC(\IR^n,\IR)$
are easily seen to be computable. 
This leads to the following corollary.

\begin{corollary}
\label{cor:outer-approximation}
For every non-empty co-c.e.\ closed set $A\In[0,1]^n$ there is a computable sequence $(A_i)_{i\in\IN}$ of bi-computable compact sets $A_i\In[-1,2]^n$
that is compactly decreasing, i.e., $A_{i+1}\Subset A_i$ for all $i\in\IN$ and such that $A=\bigcap_{i\in\IN}A_i$.
\end{corollary}

The representation of closed sets $A\In[0,1]^n$ by trees of rational complexes also has the advantage that connectedness components
of $A$ can easily be expressed in terms of the tree structure. This is made precise by the following lemma.
By $[T]:=\{p\in\IN^\IN:(\forall i)\;p|_i\in T\}$ we denote the set of {\em infinite paths} of $T$, which is also called the {\em body} of $T$.
Here $p|_i=p(0)...p(i-1)\in\IN^*$ denotes the {\em prefix} of $p$ of length $i$ for each $i\in\IN$.
According to the following lemma
there is bijection between $[T]$ and the set of connectedness components of a non-empty closed set $A\In[0,1]^n$.

\begin{lemma}[Connectedness components]
\label{lem:connected-component}
Let $(T,f)$ be an infinite tree of $n$--dimensional non-empty rational complexes and let $A\In[0,1]^n$ be the non-empty closed set represented by $(T,f)$.
Then the sets $C_p:=\bigcap_{i=0}^\infty\bigcup f(p|_i)$ for $p\in[T]$ are exactly all connectedness components of $A$
(without repetitions).
\end{lemma}
\begin{proof}
Let $A\In[0,1]^n$ be represented by $(T,f)$. 
Firstly, it is clear that every set $C_p=\bigcap_{i=0}^\infty\bigcup f(p|_i)$ is included in $A$ for $p\in[T]$.
We claim that also $\bigcup_{p\in[T]}C_p=A$. If $x\in A$, then for every $i$ there is a unique $w_i\in T\cap\IN^i$ such that $x\in\bigcup f(w_i)$.
Since $w\prefix w_i$ and $w\not=w_i$ imply $\bigcup f(w_i)\In\bigcup f(w)$, it follows that $T_x:=\{w_i:i\in\IN\}$ constitutes
an infinite finitely branching tree and by Weak K\H{o}nig's Lemma this tree has an infinite path $p$ such that $x\in C_p$.
Now we claim that $\bigcup_{p\in[T]}C_p$ is even a pairwise disjoint union. 
Let $x\in C_p\cap C_q$ for $p,q\in[T]$ with $p\not=q$. Then there is an $i\in\IN$ such that $p|_i\not=q|_i$ and 
we have $x\in\bigcup f(p|_i)\cap\bigcup f(q|_i)$. This contradicts the fact that $(T,f)$ is a tree of rational complexes.
Hence, the union $\bigcup_{p\in[T]}C_p$ is a disjoint union.
By definition of a tree of rational complexes, $\bigcup f(p|_i)$ is connected and compact for 
every $i\in\IN$. It follows that $C_p$ is connected, since the intersection of a sequence of continua is a continuum (i.e., connected and compact, 
see for instance Corollary~6.1.19 in \cite{Eng89}).
Altogether, this proves the claim.
\end{proof}

As another interesting result we can deduce from Proposition~\ref{prop:complexes} a classification of the operation 
that determines a connectedness component. We first define this operation.
For brevity, we denote by $\AA_n$ the subspace of non-empty closed subsets of $\AA_-([0,1]^n)$.

\begin{definition}[Connectedness components]
By $\Con_n:\AA_n\mto\AA_n$ we denote the map with $\Con_n(A):=\{C:C$ is a connectedness component of $A\}$ for every $n\geq 1$.
\end{definition}

We note that the Weihrauch degree of Weak K\H{o}nig's Lemma has been defined and studied in \cite{GM09,BG11,BG11a,BBP12,BGM12}.
Here we prove that the problem $\Con_n$ of finding a connectedness component of a closed set has the
same strong Weihrauch degree as Weak K\H{o}nig's Lemma for every dimension $n\geq 1$.

\begin{theorem}[Connectedness components]
\label{thm:Con}
$\Con_n\equivSW\WKL$ for $n\geq 1$.
\end{theorem}
\begin{proof}
Given a set $A\In[0,1]^n$, we can compute a tree $(T,f)$ of rational complexes that represents $A$ (together with a bound $b$ of $T$).
With the help of Weak K\H{o}nig's Lemma $\WKL$ we can find an infinite path $p\in[T]$ of $T$ (since the bound $b$ is available).
Then $C=\bigcap_{i=0}^\infty\bigcup f(p|_i)$ is a connectedness component of $A$ by Lemma~\ref{lem:connected-component}
and given $T,f,p$ we can actually compute $C\in\AA_n$.
This proves $\Con_n\leqW\WKL$ and since $\WKL$ is a cylinder (see \cite{BG11}) this even implies $\Con_n\leqSW\WKL$.

For the other direction, $\WKL\leqSW\Con_1$ we use a standard computable embedding $\iota:\{0,1\}^\IN\to[0,1]$ of Cantor space 
into the unit interval with a computable right inverse. Given a tree $T$ with infinite paths we can compute the set 
$A=[T]\in\AA_-(\{0,1\}^\IN)$ of infinite paths and hence we can also compute $\iota(A)\in\AA_1$ (see \cite{BG09}). 
Using $\Con_1$ we obtain a connectedness component $C\in\AA_1$ of $\iota(A)$.
Since $\iota(\{0,1\}^\IN)$ is totally disconnected, any connectedness component $C$ of $\iota(A)$ is actually a singleton
and hence we can compute $x$ with $C=\{x\}$ (since $[0,1]$ is compact). Hence $p=\iota^{-1}(x)$ is an infinite path in $T$.
This proves $\WKL\leqSW\Con_1$ and the higher dimensional case can be treated analogously (using the canonical embedding of $[0,1]$ into $[0,1]^n$).
\end{proof}

In \cite{LZ08a}, Le Roux and Ziegler studied computability properties of closed sets and their connectedness components.
For instance, they prove that any co-c.e.\ closed set with only finitely many connectedness components has only co-c.e.\ closed
connectedness components and any co-c.e.\ closed set without co-c.e.\ closed connectedness components has continuum cardinality
many connectedness components. This can easily be deduced from the previous theorem as well as many other properties of connectedness components.
For instance, it is well known that there exists a computable tree with countably many infinite paths and a unique non-isolated infinite path that is
not even limit computable (see Theorem~2.18 in \cite{CDJS93}).
This implies the following result, which resolves the Open Question 4.10 in \cite{LZ08a}.

\begin{corollary}
There exists a non-empty co-c.e.\ closed set $A\In[0,1]$ with only countably many connectedness components
one of which is not co-c.e.\ closed (and it is not even the set of accumulation points of a computable sequence).
\end{corollary}

We mention that a closed set is the set of accumulation points of a computable sequence if and only if it has a limit
computable name (i.e., if it is co-c.e.\ closed in the halting problem, see \cite{LZ08a,BGM12}).
Another consequence of Lemma~\ref{lem:connected-component} using the Low Basis Theorem (see \cite{Soa87}) 
is that every co-c.e.\ closed set has a low connectedness component in the sense that it is low in the space $\AA_-([0,1]^n)$.
We describe this result in the special case of the representation of closed sets considered here.

\begin{corollary}
Let $A\In[0,1]^n$ be co-c.e.\ closed. Then there is a computable sequence $(A_i)_{i\in\IN}$ of bi-computable closed
sets $A_i\In[0,1]^n$ and a low $p\in\IN^\IN$ such that
$\bigcap_{i=0}^\infty A_{p(i)}$ is a connectedness component of $A$ (which is then, in particular, the set of accumulation 
points of a computable sequence).
\end{corollary}

We close this section by mentioning that one can use the representation of closed sets by trees of rational complexes in
order to prove that the function $(A,x)\mapsto C$ that maps any non-empty closed set $A$ together with a point $x\in A$
to the connectedness component $C$ of $A$ that contains $x$ is computable. The point $x$ guides the path in the
tree of rational complexes that one has to take. This result was already proved in \cite{LZ08a}. We formulate
a non-uniform corollary here.

\begin{corollary}
Every connectedness component of a co-c.e.\ closed set $A\In[0,1]^n$ that contains a computable point $x\in[0,1]^n$
is itself co-c.e.\ closed.
\end{corollary}

We note that in the one-dimensional case an inverse holds true: every non-empty connected co-c.e.\ closed set $A\In[0,1]$
contains a computable point. However, the analogue statement is no longer true from dimension two upwards (see Corollary~\ref{cor:Orevkov-Baigger}).
Further interesting results on connected co-c.e.\ closed sets can be found in \cite{Kih12}.

\section{Brouwer's Fixed Point Theorem and Connected Choice}
\label{sec:BFT}

In this section, we want to prove that the Brouwer Fixed Point Theorem is computably equivalent to 
connected choice for any fixed dimension. We first define these two operations. By $\CC(X,Y)$ we denote
the {\em set of continuous functions} $f:X\to Y$ and for short we write $\CC_n:=\CC([0,1]^n,[0,1]^n)$.

\begin{definition}[Brouwer Fixed Point Theorem]
By $\BFT_n:\CC_n\mto[0,1]^n$ we denote the operation defined by $\BFT_n(f):=\{x\in[0,1]^n:f(x)=x\}$
for $n\in\IN$.
\end{definition}

We note that $\BFT_n$ is well-defined, i.e., $\BFT_n(f)$ is non-empty for all $f$, 
since by the Brouwer Fixed Point Theorem every $f\in\CC_n$ admits
a fixed point $x$, i.e., with $f(x)=x$. 
We can also consider the infinite dimensional version of the Brouwer Fixed Point Theorem on the Hilbert cube $[0,1]^\IN$,
which is represented by the map $\BFT_\infty:\CC([0,1]^\IN,[0,1]^\IN)\mto[0,1]^\IN$ with $\BFT_\infty(f):=\{x\in[0,1]^\IN:f(x)=x\}$.
This can also be seen as a special case of the Schauder Fixed Point Theorem and hence $\BFT_\infty$ is well-defined too.
We now define connected choice.

\begin{definition}[Connected choice]
By $\ConC_n:\In\AA_n\mto[0,1]^n$ we denote the operation defined by $\ConC_n(A):=A$ for all non-empty 
connected closed $A\In[0,1]^n$ and $n\in\IN$.
We call $\ConC_n$ {\em connected choice (of dimension $n$)}.
\end{definition}

Hence, connected choice is just the restriction of closed choice to connected sets.
We also use the following notation for the set of fixed points of a functions $f\in\CC_n$.

\begin{definition}[Set of fixed points]
By $\Fix_n:\CC_n\to\AA_n$ we denote the function with $\Fix_n(f):=\{x\in[0,1]^n:f(x)=x\}$.
\end{definition}

It is easy to see that $\Fix_n$ is computable, since $\Fix_n(f):=(f-\id|_{[0,1]^n})^{-1}\{0\}$ and it is well-known
that closed sets in $\AA_n$ can also be represented as zero sets of continuous functions (see \cite{BW99,BP03}).
We note that the Brouwer Fixed Point Theorem can be decomposed to $\BFT_n\supseteq\ConC_n\circ\Con_n\circ\Fix_n$.

The main result of this section is that the Brouwer Fixed Point Theorem and connected choice
are (strongly) equivalent for any fixed dimension $n$ (see Theorem~\ref{thm:BFT} below).
An important tool for both directions of the proof is the representation of closed sets by trees
of rational complexes.
We start with the direction $\ConC_n\leqSW\BFT_n$ that can be seen as a uniformization of an earlier
construction of Baigger \cite{Bai85} that is in turn built on results of Orevkov \cite{Ore63}.

We first formulate a stronger conclusion that we can derive from Proposition~\ref{prop:complexes} in case of connected sets.
In order to express these stronger conclusions we first recall the notion of effective pathwise 
connectedness as it was introduced in \cite{Bra08}. Essentially, a set is called {\em effectively pathwise connected},
if for every two points in the set we can compute a path that connects these two points entirely within this set.\footnote{We mention that the Warsaw circle is an example of a set that is 
pathwise connected but not effectively so, not even with respect to some oracle.}
We need a uniform such notion for sequences.

\begin{definition}[Effectively pathwise connected]
Let $(A_i)_{i\in\IN}$ be a sequence of non-empty closed sets $A_i\In\IR^n$. 
Then $(A_i)_{i\in\IN}$ is called {\em pathwise connected},
if there is a function $U:\In\IN\times\IR^n\times\IR^n\mto\CC([0,1],\IR^n)$, such that
for every $p\in U(i,x,y)$ with $x,y\in A_i$ we obtain $p(0)=x$, $p(1)=y$ and $\range(p)\In A_i$.
Such a $U$ is called a {\em witness of pathwise connectedness}.
If there is a computable such witness $U$, then $(A_i)_{i\in\IN}$ is called {\em effectively pathwise connected}.
\end{definition}
 
If a (name of a realizer of a) witness $U$ of pathwise connectedness of $(A_i)_{i\in\IN}$ can be computed from $A$, then 
we say that $(A_i)_{i\in\IN}$ is {\em pathwise connected uniformly in $A$}.
We note that any rational complex $R\In\CQ^n$ is connected and also automatically pathwise connected,
due to the simple structure of such complexes.
It is easy to see that there is a computable map that maps any rational complex $R\in\CQ^n$ to a witness of
pathwise connectedness of $\bigcup R$.
By $d(A,B):=\inf_{a\in A,b\in B}||a-b||$ we denote the {\em minimal distance} between sets $A,B\In\IR^n$.
We note that $d(A,B^\cc)>0$ is equivalent to $A\Subset B$ for non-empty compact $A,B\In\IR^n$.

\begin{proposition}[Connected sets]
\label{prop:connected-path-connected}
Given a non-empty connected closed set $A\In[0,1]^n$ 
we can compute sequences of distance functions $(d_{A_i})_{i\in\IN}$ and $(d_{A_i^\cc})_{i\in\IN}$
for non-empty closed sets $A_i\In[-1,2]^n$ such that:
\begin{enumerate}
\item $A=\bigcap_{i=0}^\infty A_i$,
\item $d(A_{i+1},A_i^\cc)>0$ for all $i\in\IN$,
\item $(A_i)_{i\in\IN}$ is pathwise connected uniformly in $A$.
\end{enumerate}
\end{proposition}
\begin{proof}
Given a non-empty connected closed $A\In[0,1]^n$ we can compute an infinite tree of non-empty rational complexes $(T,f)$ that
represents $A$ by Proposition~\ref{prop:complexes}.
Since $A$ is connected, $A$ is its only connectedness component and by Lemma~\ref{lem:connected-component} 
there is exactly one infinite path $p\in[T]$. 
If we can find this path, then $A_i:=\bigcup f(p|_i)$ is a sequence of closed sets $A_i\In[-1,2]^n$
with $A_{i+1}\Subset A_i$ for all $i$, which implies $d(A_{i+1},A_i^\cc)>0$ and $A=\bigcap_{i=0}^\infty A_i$. 
Since $f(p|_i)$ is a rational complex, it is straightforward how to determine $d_{A_i}$ and $d_{A_i^\cc}$, given
this complex and since $\bigcup f(p|_i)$ is connected it is also automatically pathwise connected and
a witness $U$ for pathwise connectedness can be easily computed. 

It remains to show how we can compute the unique infinite path $p$ in $T$.
For each fixed $i$ there are only finitely many words $w_0,...,w_k\in T\cap\IN^i$ and 
due to connectedness of $A$ and $\bigcup f(w_j)$ and the fact that all the $\bigcup f(w_j)$
are pairwise disjoint, it follows that there is exactly one such $w_a$ with $A\In\bigcup f(w_a)$. 
Due to compactness of the $\bigcup f(w_j)$ all the other $w_j$ with $j\not=a$ will eventually
be covered by negative information given as input for $A$ and if this happens it can be recognized. 
Hence, one just needs to wait until all the $\bigcup f(w_j)$ except one are covered by
negative information in order to identify $w_a$. Then $w_a\prefix p$ and by a repetition of this
procedure for each $i$ one can compute $p$.
\end{proof}

Now we use Proposition~\ref{prop:connected-path-connected} to prove that every non-empty connected closed set $A\In[0,1]^n$ can be 
effectively translated into a continuous function $f\in\CC_n$ that has all its fixed points in $A$.
The idea is to compute a compactly decreasing sequence $(A_i)_{i\in\IN}$ of closed sets according to the previous
proposition together with points $x_i\in A_i$ and paths $p_i$ in $A_i$ that connect $x_{i+1}$ with $x_i$.
In some sense we then use these paths like Ariadne's thread in order to construct a function $f$ that is a modified identity
with all fixed points shifted towards $A$ along the given paths.
By $||f||:=\sup_{x\in[0,1]^n}||f(x)||$ we denote the supremum norm for continuous functions $f:[0,1]^n\to[0,1]^n$.

\begin{lemma}
\label{lem:CC-BFT}
$\ConC_n\leqSW\BFT_n$ for all $n\geq1$.
\end{lemma}
\begin{proof}
Given a non-empty closed and connected set $A\In[0,1]^n$, we will compute a function $f\in\CC_n$ such that 
all fixed points of $f$ are included in $A$.
Firstly, we compute the sequences $(d_{A_i})_{i\in\IN}$ and $(d_{A_i^\cc})_{i\in\IN}$
according to Proposition~\ref{prop:connected-path-connected}.

Without loss of generality, we can assume that $A\In[2^{-3},1-2^{-3}]^n$ and all $A_i\In[2^{-4},1-2^{-4}]^n=:Q$. 
This can always be achieved using a suitable computable homeomorphism $T:[-1,2]^n\to[2^{-4},1-2^{-4}]^n$ that is applied
to all input data and afterwards the fixed point $x$ that is found is transferred back by $T^{-1}(x)$.

Since we can compute the sequences of distance functions $(d_{A_i})_{i\in\IN}$
we can also find a sequence of points $(x_i)_{i\in\IN}$ with $x_i\in A_i$ for all $i\in\IN$. 
Since $(A_i)_{i\in\IN}$ is pathwise connected uniformly in $A$, 
we can also compute a sequence $(p_i)_{i\in\IN}$ of continuous functions $p_i:[0,1]\to[0,1]^n$
such that $p_i(0)=x_{i+1}$, $p_i(1)=x_i$ and $\range(p_i)\In A_i$. 
We can also uniformly compute a sequence $(D_i)_{i\in\IN}$ of functions $D_i:[0,1]^n\to[0,1]$ defined by
\[D_i(x):=\frac{d(x,A_{i+1})}{d(x,A_{i+1})+d(x,A_{i}^\cc)}\]
for all $x\in[0,1]^n$ and $i\in\IN$. Since $d(A_{i+1},A_i^\cc)>0$ for all $i\in\IN$, it follows
that the denominator is always non-zero and hence the functions $D_i$ are well-defined.
We obtain $D_i(x)=0\iff x\in A_{i+1}$ and $D_i(x)=1\iff x\in\overline{A_i^\cc}$.

We now compute a continuous function $f:[0,1]^n\to[0,1]^n$ with $\BFT_n(f)\In A$. 
The function $f$ will be defined as $f:=\id+2^{-4}\sum_{i=0}^\infty g_i$ using further continuous functions $g_i$.
As an abbreviation we write $G_i:=\sum_{j=0}^{i}g_j$ for the partial sums.
We also use the abbreviations $C_n:=\sum_{i=n}^\infty2^{-3i-1}$ and we note that $C_n\leq2^{-3n}$ for all $n\in\IN$.
We start with
\[g_0(x):=\left\{\begin{array}{ll}
2^{-1}\frac{\displaystyle x_2-x}{\displaystyle||x_2-x||}d(x,A_1) & \mbox{if $x\not\in A_2$}\\
0 & \mbox{otherwise}
\end{array}\right.\]
for all $x\in[0,1]^n$. 
In the next step we define inductively
\[g_{i+1}(x):=\left\{\begin{array}{ll}
2^{-3i-4}\frac{\displaystyle G_{i}(x)}{\displaystyle||G_{i}(x)||} & \mbox{if $x\not\in A_{i+1}$}\\[0.2cm]
2^{-3i-4}\frac{\displaystyle p_{i+2}(D_{i+1}(x))-x}{\displaystyle||p_{i+2}(D_{i+1}(x))-x||}D_{i+1}(x) & \mbox{if $x\in A_{i+1}\setminus A_{i+2}$}\\[0.2cm]
0 & \mbox{if $x\in A_{i+2}$}
\end{array}\right.\]
for all $x\in[0,1]^n$ and $i\in\IN$.

We first prove that all $g_i$ and $\sum_{i=0}^\infty g_i(x)$ are well-defined and
\begin{eqnarray}
\label{eqn:fixed-point}
x\in A=\bigcap_{i=0}^\infty A_i\iff\sum_{i=0}^\infty g_i(x)=0\iff f(x)=x.
\end{eqnarray}
The second equivalence follows immediately from the definition of $f$ (once we know that the $g_i$ and $\sum_{i=0}^\infty g_i$ are well defined).
If $x\in\bigcap_{i=0}^\infty A_i$, then it follows immediately that $g_i(x)=0$ for all $i$ and hence $\sum_{i=0}^\infty g_i(x)=0$.
If $x\not\in\bigcap_{i=0}^\infty A_i$, then there is a minimal $m\in\IN$ with $x\not\in A_m$, since $(A_i)_{i\in\IN}$ is decreasing.
If $m\in\{0,1\}$, then $x\not\in A_1$ and hence $x\not\in A_2$. Since $x_2\in A_2$ it follows that $||x_2-x||\not=0$.
We also obtain $d(x,A_1)>0$ and thus $g_0(x)\not=0$. This implies
\begin{eqnarray}
\label{eqn:fixed-point-0}\qquad
\sum_{i=0}^\infty g_i(x)=g_0(x)+\sum_{i=1}^\infty 2^{-3i-1}\frac{g_0(x)}{||g_0(x)||}=\frac{\displaystyle x_2-x}{\displaystyle||x_2-x||}(2^{-1}d(x,A_1)+C_1)\not=0.
\end{eqnarray}
If $m>1$, then $x\in A_{m-1}\setminus A_m$ and it follows that $g_i(x)=0$ for $i\leq m-2$.
Since $\range(p_m)\In A_m$ and $x\not\in A_m$, it follows that $||p_{m}(D_{m-1}(x))-x||\not=0$.
We also have $D_{m-1}(x)\not=0$ and hence $g_{m-1}(x)\not=0$. This implies
\[\sum_{i=0}^\infty g_i(x)=g_{m-1}(x)+\sum_{i=m}^\infty g_i(x)=
\frac{\displaystyle p_{m}(D_{m-1}(x))-x}{\displaystyle||p_{m}(D_{m-1}(x))-x||}(2^{-3m+2}D_{m-1}(x)+C_m)\not=0.\]
These two cases together prove the first equivalence in (\ref{eqn:fixed-point}) together with the fact that all $g_i$
and $\sum_{i=0}^\infty g_i(x)$ are well-defined. 
We can also conclude from Equation (\ref{eqn:fixed-point}) that $A$ is exactly the set of fixed point of $f$.

Next we want to show that by $f:=\id+2^{-4}\sum_{i=0}^\infty g_i$ actually a continuous function of type $f:[0,1]^n\to[0,1]^n$ is defined.
We show that $f([0,1]^n)\In[0,1]^n$. If $x\in[0,1]^n\setminus A_0$, then Equation~(\ref{eqn:fixed-point-0})
implies 
$f(x)=x+2^{-4}\frac{x_2-x}{||x_2-x||}(2^{-1}d(x,A_1)+C_1)$, which means that $f$ moves $x$ 
towards $x_2\in A_0\In Q$ and, in particular, $f(x)\in[0,1]^n$.
If $x\in A_0\In Q$, then $f(x)=x+2^{-4}\sum_{i=0}^\infty g_i(x)\in[0,1]^n$ since  
$||g_i||=\sup_{x\in[0,1]^n}||g_i(x)||\leq2^{-3i-1}$ and hence $||2^{-4}\sum_{i=0}^\infty g_i(x)||\leq2^{-4}C_0\leq2^{-4}$.
Now we prove that $f$ is also continuous.
First we show that each function $g_{i}$ is continuous. We start with $g_0$. 
If $x$ approaches $\partial A_2$ from the outside, then eventually $d(x,A_1)=0$ and hence $g_0(x)=0$. 
This means that $g_0$ continuous. We now continue with $g_{i+1}$.
If $x\in\partial A_{i+1}=\partial\overline{A_{i+1}^\cc}$,
then $D_{i+1}(x)=1$ and hence $p_{i+2}(D_{i+1}(x))=x_{i+2}$ and we obtain
\[g_{i+1}(x)=2^{-3i-4}\frac{\displaystyle p_{i+2}(D_{i+1}(x))-x}{\displaystyle||p_{i+2}(D_{i+1}(x))-x||}D_{i+1}(x)=2^{-3i-4}\frac{\displaystyle x_{i+2}-x}{\displaystyle||x_{i+2}-x||}.\]
If, on the other hand, $x$ approaches $\partial A_{i+1}$ from the outside of $A_{i+1}$, then $D_i(x)\to 0$ and $x$ is eventually in 
$A_{i}$ and hence $g_{j}(x)=0$ for $j\leq i-1$ and $G_i=g_i$.
In case $i>0$ we use $D_i(x)\to 0$ in order to conclude
\[g_{i+1}(x)=2^{-3i-4}\frac{\displaystyle G_{i}(x)}{\displaystyle||G_{i}(x)||}=
2^{-3i-4}\frac{\displaystyle p_{i+1}(D_{i}(x))-x}{\displaystyle||p_{i+1}(D_{i}(x))-x||}\to2^{-3i-4}\frac{\displaystyle x_{i+2}-x}{\displaystyle||x_{i+2}-x||}.\]
In case of $i=0$ we obtain
\[g_1(x)=2^{-4}\frac{\displaystyle G_{0}(x)}{\displaystyle||G_{0}(x)||}=2^{-4}\frac{x_2-x}{||x_2-x||}.\]
Finally, if $x$ approaches $\partial A_{i+2}$ from the outside, then $D_{i+1}(x)\to 0$ and $x$ is eventually in $A_{i+1}$.
Hence 
\[g_{i+1}(x)=2^{-3i-4}\frac{\displaystyle p_{i+2}(D_{i+1}(x))-x}{\displaystyle||p_{i+2}(D_{i+1}(x))-x||}D_{i+1}(x)\to0.\]
Altogether, this proves that the case distinction in the definition of $g_{i}$ is continuous and it is also computable since
\begin{enumerate}
\item $x\not\in A_{i+1}\iff D_i(x)>0$,
\item $x\in A_{i+1}\setminus A_{i+2}\iff D_i(x)=0$ and $D_{i+1}(x)>0$,
\item $x\in A_{i+2}\iff D_{i+1}(x)=0$.
\end{enumerate}
Hence all the functions $g_i$ and $f$ are continuous and can be uniformly computed in the input $A$.
We also obtain $\BFT_n(f)=A$ by Equation (\ref{eqn:fixed-point}), which proves $\ConC_n\leqSW\BFT_n$.
\end{proof}

We note that the proof shows more than necessary. We only need that $\BFT_n(f)\In A$ and we even obtain equality.

For the other direction $\BFT_n\leqSW\ConC_n$ of the reduction we uniformize ideas presented by J.S.\ Miller \cite[Section~2.3]{Mil02a}.
He proved the following result in terms of simplicial complexes. 
We note that rational complexes can be effectively converted into corresponding simplicial complexes.

\begin{proposition}[Topological index, J.S.\ Miller 2002]
\label{prop:index}
There is a computable topological index function $\ind:\In\CC_n\times\CQ^n\to\IZ$ such that for all $f\in\CC_n$ and $S,S_1,S_2\in\CQ^n$
such that $f$ has no fixed points on $\partial\bigcup S_1$ and $\partial\bigcup S_2$ the following holds:
\begin{enumerate}
\item $\ind(f,S)$ is defined if and only if $f(x)\not=x$ for all $x\in\partial\bigcup S$.
\item $\ind(f,S)\not=0$ implies that $f(x)=x$ for some $x\in\bigcup S$.
\item $\ind(f,\{[0,1]^n\})\not=0$.
\item If $\{x\in \bigcup S_1:f(x)=x\}=\{x\in \bigcup S_2:f(x)=x\}$, then one obtains $\ind(f,S_1)=\ind(f,S_2)$.
\item If $\bigcup S_1$ and $\bigcup S_2$ are disjoint, then $\ind(f,S_1\cup S_2)=\ind(f,S_1)+\ind(f,S_2)$.
\end{enumerate}
\end{proposition}

The proof of this result uses simplicial homology theory and, more specifically, the local topological degree.
The effectivization follows the lines of classically known results in algebraic topology.
Computability aspects of homology have also been studied by
in a discrete setting by Kaczynski et.\ al.\ \cite{KMM04} and in the context of computable analysis by Collins \cite{Col08,Col09}. 
We essentially use Miller's ideas to reduce the Brouwer Fixed Point Theorem uniformly to connected choice.
First we prove that the map $\Con_n\circ\Fix_n$ is computable (which might be surprising in light of Theorem~\ref{thm:Con}).

\begin{proposition}
\label{prop:Con-Fix}
$\Con_n\circ\Fix_n:\CC_n\mto\AA_n$ is computable for all $n\in\IN$.
\end{proposition}
\begin{proof}
Given a continuous function $f\in\CC_n$ we can easily compute the set of fixed points $A:=\{x\in[0,1]^n:f(x)=x\}\in\AA_n$.
Using Proposition~\ref{prop:complexes} we can compute a tree $(T,f)$ of rational complexes that represents $A$.
Using Proposition~\ref{prop:index} we can now identify an infinite path $p$ in $T$ and hence
by Lemma~\ref{lem:connected-component} a connectedness component $C$ of $A$. 

We start with the empty node $\varepsilon$ in $T$. 
Given a node $w\in T$, we construct an extension $wi\in T$ that is part of an infinite path as follows.
Let us assume that $S_0=f(w0),...,S_k=f(wk)$ are the rational complexes that we need to consider. 
Due to the definition of a tree of rational complexes we know that $\bigcup S_i\cap\bigcup S_j=\emptyset$ for $i\not=j$.
Since $A\Subset\bigcup_{j=0}^k\bigcup S_j$, it is clear that $f$ cannot have any fixed point on
any of the boundaries $\partial\bigcup S_j$, and hence we can compute the indexes $\ind(f,S_0),...,\ind(f,S_k)$ by Proposition~\ref{prop:index} (1).
One of them, say $\ind(f,S_i)$, must be different from $0$, as one can see inductively using Proposition~\ref{prop:index} (3)-(5).
By Proposition~\ref{prop:index} (2), this means that $f$ has a fixed point in $S_i$, which means
that $A\cap\bigcup S_i\not=\emptyset$. We use this $wi$ as an extension of $w$ and we proceed
inductively in the same manner.

Altogether, this algorithm produces an infinite path $p$ of $T$ and hence we can compute
the connected component $C:=\{\bigcap_{i=0}^\infty\bigcup f(p|_i):p\in[T]\}\in\AA_n$ of $A$ by Lemma~\ref{lem:connected-component}.
This shows that $\Con_n\circ\Fix_n$ is computable.
\end{proof}

Since $\BFT_n(f)\supseteq\ConC_n\circ\Con_n\circ\Fix_n(f)$ we can directly conclude $\BFT_n\leqSW\ConC_n$ for all $n$.
Together with Lemma~\ref{lem:CC-BFT} we obtain the following theorem.

\begin{theorem}[Brouwer Fixed Point Theorem]
\label{thm:BFT}
$\BFT_n\equivSW\ConC_n$ for all $n\in\IN$.
\end{theorem}

It is easy to see that in general the Brouwer Fixed Point Theorem and connected choice
are not independent of the dimension. In case of $n=0$ the space $[0,1]^n$ is the
one-point space $\{0\}$ and hence $\BFT_0\equivSW\ConC_0$ are both computable. 
In case of $n=1$ connected choice was already studied in \cite{BG11a} and it was
proved that it is equivalent to the Intermediate Value Theorem $\IVT$ (see Definition~6.1 and Theorem~6.2 in \cite{BG11a}).

\begin{corollary}[Intermediate Value Theorem]
\label{cor:IVT}
$\IVT\equivSW\BFT_1\equivSW\ConC_1$.
\end{corollary}

It is also easy to see that the Brouwer Fixed Point Theorem $\BFT_2$ in dimension two is 
more complicated than in dimension one. For instance, it is known that the Intermediate Value
Theorem $\IVT$ always offers a computable function value for a computable input, whereas
this is not the case for the Brouwer Fixed Point Theorem $\BFT_2$ by Baigger's counterexample
\cite{Bai85}. We continue to discuss this topic in Section~\ref{sec:dimension}.

Here we point out that Proposition~\ref{prop:Con-Fix} implies that the fixed point set $\Fix_n(f)$ 
of every computable function $f:[0,1]^n\to[0,1]^n$ has a co-c.e.\ closed connectedness component.
J.S.\ Miller observed that also any co-c.e.\ closed superset of such a set is the
fixed point set of some computable function and the following result is a uniform version
of this observation.
We denote by $(f,g):\In X\mto Y\times Z$ the {\em juxtaposition} of two functions $f:\In X\mto Y$
and $g:\In X\mto Z$, defined by $(f,g)(x)=(f(x),g(x))$.

\begin{theorem}[Fixability]
\label{thm:fixable}
$(\Fix_n,\Con_n\circ\Fix_n)$ is computable and has a multi-valued computable right inverse for all $n\in\IN$.
\end{theorem}
\begin{proof}
It follows directly from Proposition~\ref{prop:Con-Fix} and the fact that $\Fix_n$ is computable
that $(\Fix_n,\Con_n\circ\Fix_n)$ is computable for all $n\in\IN$.
We now describe how a right inverse $R:\In\AA_n\times\AA_n\mto\CC_n$ can be computed.
Firstly, given $(A,C)$ such that $A\in\AA_n$ and $C$ is a connectedness component of $A$,
we can find some $f\in\CC_n$ such that $\Fix_n(f)=C$ following the algorithm that is specified
in the proof of Lemma~\ref{lem:CC-BFT}. We can also find a continuous $g:[0,1]^n\to[0,1]$ such that
$g^{-1}\{0\}=A$ (see \cite{BW99}). Then we can also compute a continuous $h$ with
\[h(x):=(1-g(x))x+f(x)g(x)\]
and since this is a convex combination of $\id$ and $f$, it follows that $h$ is actually a continuous function $h:[0,1]^n\to[0,1]^n$. 
Finally, 
\[h(x)=x\iff(f(x)-x)g(x)=0\iff x\in C\cup A=A.\]
That is, $\Fix_n(h)=A$. Hence the function $R$ with $(A,C)\mapsto h$ is a suitable computable
right inverse of $(\Fix_n,\Con_n\circ\Fix_n)$.
\end{proof}

Roughly speaking, a closed set $A\in\AA_n$ together with one of its connectedness components is as
good as a continuous function $f\in\CC_n$ with $A$ as set of fixed points.
As a non-uniform corollary we obtain immediately Miller's original result \cite[Theorem~2.6.1]{Mil02a}.

\begin{corollary}[Fixable sets, J.S.\ Miller 2002]
A set $A\In[0,1]^n$ is the set of fixed points of a computable function $f:[0,1]^n\to[0,1]^n$
if and only if it is non-empty and co-c.e.\ closed and contains a co-c.e.\ closed connectedness component.
\end{corollary}

We can also derive other interesting results from Theorem~\ref{thm:fixable}.
For instance we can derive an upper bound on how complex a continuous functions needs to be that 
has an arbitrary given non-empty co-c.e.\ closed set as fixed point set.

\begin{corollary}
Let $A\In[0,1]^n$ be a non-empty co-c.e.\ closed set. Then there is a
continuous function $f:[0,1]^n\to[0,1]^n$ that is low as a point in $\CC_n$ and has $A$ as fixed point set.
\end{corollary}

This result follows from an application of the Uniform Low Basis Theorem \cite[Theorem~8.3]{BBP12}
since $\Fix_n$ has a right inverse that is reducible to $\WKL$ by Theorems~\ref{thm:fixable}
and \ref{thm:Con}.\footnote{We mention that a function $f$ that is low as a point in $\CC_n$
is not necessarily low as a function in the sense that $f\leqSW\L$ (where $\L=\J^{-1}\circ\lim$ is the composition of
the inverse of the Turing jump $\J$ and the limit operation), but one only obtains $f\leqW\L$ here
(see \cite{BBP12,BGM12} for a discussion of low functions).}

\section{Lipschitz Continuity}
\label{sec:Lipschitz}

In this section we want to discuss the question whether Lipschitz continuity of a function $f:[0,1]^n\to[0,1]^n$
simplifies finding fixed points in any way, compared to a function $f$ that is just continuous.\footnote{We would like to thank Ulrich Kohlenbach for raising this question.}
We recall that a function $f:[0,1]^n\to[0,1]^n$ is called {\em Lipschitz continuous} with constant $L\geq0$, if
\[||f(x)-f(y)||\leq L\cdot||x-y||\]
holds for all $x,y\in[0,1]^n$. As before, $||\;||$ denotes the maximum norm on Euclidean space.
We are going to prove in this section that a Lipschitz constant $L>1$ as an extra constraint does not simplify finding fixed points.
We first need a refined version of Proposition~\ref{prop:connected-path-connected}.

\begin{proposition}
\label{prop:largedistanceshortpath}
Given a non-empty connected closed set $A \subseteq [0,1]^n$ we can compute sequences of distance functions $(d_{A_i})_{i \in \IN}$ and $(d_{A^\cc_i})_{i \in\IN}$
for non-empty closed sets $A_i \subseteq [-1, 2]^n$, a sequence $(x_i)_{i \in\IN}$ of points in $[-1, 2]^n$ and a sequence 
$(p_i)_{i \in\IN}$ of paths $p_i:[0,1]\to[-1, 2]^n$ such that for all $i\in\IN$:
\begin{enumerate}
\item $\bigcap_{i=0}^\infty A_i = A$,
\item $d(A_{i+1}, A_i^\cc) \geq 2^{-i-1}$,
\item $x_i \in A_i$,
\item $\range(p_i)\In A_i$ and $p_i(0) = x_i$, $p_i(1) = x_{i+1}$,
\item $d(\range(p_i), A_i^\cc) \geq 2^{-i}$
\item $p_i:[0,1]\to[-1,2]^n$ is Lipschitz continuous with constant $L=1$.
\end{enumerate}
\begin{proof}
We start as in the proof of Proposition~\ref{prop:connected-path-connected} with a tree $(T,f)$ of rational complexes that represent $A$
and from which we compute sequences of distance functions $(d_{A_i})_{i \in \IN}$ and $(d_{A^\cc_i})_{i \in\IN}$ satisfying condition (1).
Then we compute a sequence of points $(x_i)_{i\in\IN}$ and paths $(p_i)_{i\in\IN}$ linking them satisfying conditions (3) and (4) as in the proof of Lemma~\ref{lem:CC-BFT}. 

The construction of the $p_i$ allows us to choose a constant--speed parameterization, i.e., a $p_i$ that is Lipschitz continuous with constant $L_i \in \mathbb{N}$, and moreover 
we can compute a sequence $(L_i)_{i\in\IN}$ of corresponding constants.
Now for any $(i, j)\in\IN^2$ with $j < L_i$, define $A_{i,j} := A_i$, $x_{i,j} := p_i(L_i^{-1}\cdot j)$ and $p_{i,j}(t) := p_i(L_i^{-1}\cdot(j + t))$ for $t\in[0,1]$. 
The purpose of these refinements is to obtain $p_{i,j}$ that are Lipschitz continuous for constant $L=1$.
Now we can determine new sequences $(A'_i)_{i\in\IN}$, $(x'_i)_{i\in\IN}$, $(p'_i)_{i\in\IN}$ by a lexicographic ordering of the double sequences 
$(A_{i,j})_{i\in\IN,j<L_i}$, $(x_{i,j})_{i\in\IN,j<L_i}$ and $(p_{i,j})_{i\in\IN,j<L_i}$, respectively.
These are clearly computable from the original ones. Moreover, the conditions (1), (3) and (4) remain unaffected, while condition (6) is now satisfied, too for the sequence $(A'_i)_{i\in\IN}$.

In order to satisfy condition (2), the construction employed in the proof of Proposition~\ref{prop:complexes} to obtain a tree (collapsing to a single path here) of rational complexes as a name for a closed set is reused. 
Firstly, we determine an enumeration of rational balls $B(c_i,r_i)$ such that $A=Q\cap(\bigcup_{i=0}^\infty B(c_i,r_i))^\cc$ with $Q:=[-1,2]^n$
with the additional property that $B(c_i,r_i)\In (A'_i)^\cc$. Now we construct a new tree $(T',f')$ of rational complexes that represents $A$ following the algorithm 
in the proof of Proposition~\ref{prop:complexes} with this particular enumeration of balls $B(c_i,r_i)$ and we use again the method in the proof of  Proposition~\ref{prop:connected-path-connected}
to obtain sequences of distance functions $(d_{A''_i})_{i \in \IN}$ and $(d_{(A_i'')^\cc})_{i \in\IN}$
for non-empty closed sets $A''_i \subseteq [-1, 2]^n$.
The additional property $B(c_i,r_i)\In (A'_i)^\cc$ guarantees that the new sets are supersets of the original ones, 
i.e., $A'_i\In A''_i$ for all $i\in\IN$. The extra margin of $2^{-i}$ provided by step (3) of the construction in the proof of Proposition~\ref{prop:complexes}
even guarantees that $d(A'_i,(A''_i)^\cc)\geq2^{-i}$ and hence, in particular, $d(\range(p_i), (A''_i)^\cc) \geq 2^{-i}$. 
An inspection of step (2) of that construction reveals that we
also obtain $d(A''_{i+1},(A''_i)^\cc)\geq2^{-i-1}$ for all $i\in\IN$.
Hence, the sequence $(A''_i)_{i\in\IN}$ still satisfies the corresponding 
conditions (1), (3), (4) and (6) and it additionally also satisfies conditions (2) and (5).
\end{proof}
\end{proposition}

With some extra calculations we can now prove a refined version of Lemma~\ref{lem:CC-BFT}
for arbitrary Lipschitz constant $L>1$.	

\begin{theorem}[Lipschitz continuity]
\label{thm:Lipschitz}
Given a non-empty connected closed set $A\In[0,1]^n$ and a real number $L>1$
we can compute a continuous function ${f:[0,1]^n\to[0,1]^n}$ that is Lipschitz continuous with constant $L$
and such that $A$ is the set of fixed points of $f$. 
\end{theorem}
\begin{proof}
Let $0<\varepsilon<1$. 
Suppose we can find a continuous function $f$ that is Lipschitz continuous with constant $L>0$ and has $A$ as set of fixed points.
Then we can compute the function $\id + \frac{\varepsilon}{1 + L}(f - \id)$ that has the same fixed points as $f$ and is Lipschitz continuous with constant $1+\varepsilon$.
Hence, it is sufficient to prove that we can compute a function $f$ that is Lipschitz continuous with constant $L=6$.
We prove that if the construction of the proof of Lemma~\ref{lem:CC-BFT} is carried out with the conditions provided by Proposition~\ref{prop:largedistanceshortpath}, 
then the resulting function $f$ is such a function. 

First, we provide a simplified expression for $f(x)$. We use the abbreviation $P_x := \frac{p_{n+1}D_n(x) - x}{||p_{n+1}D_n(x) - x||}$.
If $x \in A$, then $f(x) = x$. If $x \in A_{n} \setminus A_{n +1}$ with $n>0$, then
\begin{eqnarray*}
f(x) %& = & x + 2^{-4}\sum_{i = 0}^\infty g_i(x) = x + 2^{-4}g_n(x) + 2^{-4}\sum_{i = n+1}^\infty g_i(x) \\
%& = & x + 2^{-4}g_n(x) + 2^{-4}\sum_{i = n+1}^\infty \left (2^{-3i-1} \frac{\sum_{j = 0}^{i-1} g_j(x)}{||\sum_{j = 0}^{i-1} g_j(x)||} \right ) \\
& = & x + 2^{-4}g_n(x) + 2^{-4}\sum_{i = n+1}^\infty \left (2^{-3i-1} \frac{\sum_{j = n}^{i-1} g_j(x)}{||\sum_{j = n}^{i-1} g_j(x)||} \right ) \\
& = & x + 2^{-4}g_n(x) + 2^{-4}\sum_{i = n+1}^\infty \left (2^{-3i-1} \frac{g_n(x)}{||g_n(x)||} \right ) \\
& = & x + 2^{-4}g_n(x) + 2^{- 4}C_{n+1} \frac{g_n(x)}{||g_n(x)||} \\
& = & x + 2^{-4}\left (2^{-3n -1}D_{n}(x) + C_{n+1} \right) P_x.
\end{eqnarray*}

By continuity, this expression remains true for $x \in \overline{A_{n} \setminus A_{n +1}}$.

Now we want to estimate a Lipschitz constant for $f$ and we distinguish a number of cases. 

\underline{1.\ Case:} $x \in A_n \setminus A_{n+1}$ with $n>0$, $y \in A$. In this situation, we obtain $f(y)=y$ and $||x-y||\geq d(A_{n+2}, A_{n+1}^\cc) \geq 2^{-n-2}$.
We recall that $C_{n+1}\leq 2^{-3n-3}$ and we estimate: 
\begin{eqnarray*}
||f(x) - f(y)|| & = & ||x + 2^{-4}\left (2^{-3n -1}D_{n}(x) + C_{n+1} \right) P_x - y||\\
& \leq & ||x - y|| + 2^{-3n-5}||\left (D_{n}(x) + 1 \right)P_x|| \\
& \leq & ||x - y|| + 2^{-3n-4}\\
&\leq & 2||x - y||.
\end{eqnarray*}

\underline{2.\ Case:} $x, y \in \overline{A_n \setminus A_{n+1}}$ with $n>0$. We use $d(A_n^\cc, A_{n+1})\geq2^{-n-1}$ and we will need the following bound:
\begin{eqnarray*} 
&& ||D_n(x) - D_n(y)||\\
& = & \left|\left|\frac{d(x, A_{n+1})}{d(x, A_{n+1}) + d(x, A_n^\cc)} - \frac{d(y, A_{n+1})}{d(y, A_{n+1}) + d(y, A_n^\cc)}\right|\right| \\
& = & \left|\left|\frac{\left (d(x, A_{n+1}) - d(y, A_{n+1}) \right )d(y, A_n^\cc) + \left (d(y,A_n^\cc) - d(x, A_n^\cc)\right) d(y, A_{n+1})}{\left (d(x, A_{n+1}) + d(x, A_n^\cc)\right )\left (d(y, A_{n+1}) + d(y, A_n^\cc) \right)}\right|\right| \\
& \leq & \frac{||\left (d(x, A_{n+1}) - d(y, A_{n+1}) \right )d(y, A_n^\cc)|| + ||\left (d(y,A_n^\cc) - d(x, A_n^\cc)\right) d(y, A_{n+1})||}{||d(A_n^\cc, A_{n+1})||\cdot||\left (d(y, A_{n+1}) + d(y, A_n^\cc) \right)||}\\
& \leq & 2^{n+1}\cdot  \frac{d(x,y)d(y, A_n^\cc) + d(x,y) d(y, A_{n+1})}{d(y, A_{n+1}) + d(y, A_n^\cc) } \\
& \leq & 2^{n+1}||x - y||.
\end{eqnarray*}

Now we also use the abbreviation $N_x := ||p_{n+1}D_n(x) - x||$. Using the fact that $p_{n+1}$ is Lipschitz continuous with constant $1$, we obtain:
\begin{eqnarray*} 
||N_x - N_y|| 
&\leq& ||p_{n+1}D_n(x)-p_{n+1}D_n(y)||+||x-y||\\
&\leq& ||D_n(x)-D_n(y)||+||x-y||\\
&\leq& 2^{n+2}||x-y||.
\end{eqnarray*}

We note that $2^{-n-1} \leq N_x \leq 2$ since $d(\range(p_{n+1}),A_{n+1}^\cc)\geq2^{-n-1}$. 
For the same reason also $||p_{n+1}D_n(y) - x||\leq 2$.
Hence
\begin{eqnarray*} 
&& ||P_x - P_y|| \\
& = & N_x^{-1}N_y^{-1}||p_{n+1}D_n(x)N_y - xN_y - p_{n+1}D_n(y)N_x + N_xy|| \\
& \leq & 2^{2n+2}||N_y \left (p_{n+1}D_n(x) - p_{n+1}D_n(y) \right )+ (N_y - N_x)(p_{n+1}D_n(y) - x) - xN_x + N_xy|| \\
& \leq & 2^{2n+2}(N_x||x - y||+N_y||p_{n+1}D_n(x) - p_{n+1}D_n(y)||+ ||p_{n+1}D_n(y) - x||\cdot ||N_x - N_y||)  \\
& \leq & 2^{2n+3}(||x - y||  + ||D_n(x) - D_n(y)|| + ||N_x-N_y||)  \\
& \leq & 2^{2n+3}(1+2^{n+1}+2^{n+2})||x - y||\\
&\leq& 2^{3n+6}||x-y||.
\end{eqnarray*}

Since $2^{3n+1}C_{n+1} + D_n(y)\leq2$ and using the previous estimations, we finally obtain:
\begin{eqnarray*} 
&& ||f(x) - f(y)||\\ 
& \leq & ||x - y|| + 2^{-3n-5}||\left (D_{n}(x) + 2^{3n+1}C_{n+1} \right) P_x - \left (D_{n}(y) + 2^{3n+1}C_{n+1} \right) P_y|| \\
& = & ||x - y|| + 2^{-3n-5}||\left (D_n(x) - D_n(y) \right )P_x + \left (2^{3n+1}C_{n+1} + D_n(y)\right )\left (P_x - P_y\right ) || \\
& \leq & ||x - y|| + 2^{-3n-5}||D_n(x) - D_n(y)|| + 2^{-3n-4}||P_x - P_y|| \\
&\leq & (1 + 2^{-2n-4}+4)||x - y||\\
& \leq & 6||x - y||.
\end{eqnarray*}

\underline{3.\ Case:} $x, y \in [0,1]^n$ not satisfying the conditions from our first or second case. 
Without loss of generality we can assume $[0,1]^n\In A_1$.
The straight line from $x$ to $y$ either intersects $A$, or is composed of a finite number of line segments each fully included in some $\overline{A_i \setminus A_{i+1}}$ with $i>0$. 
In the former case, pick some $z$ from the intersection of the line and $A$. 
We obtain with the help of the estimations above 
\[||f(x) - f(y)|| \leq ||f(x) - f(z)|| + ||f(z) - f(y)|| \leq 2||x - z|| + 2||z - y|| = 2||x - y||.\] 
In the latter case, we obtain $||f(x)-f(y)||\leq 6||x-y||$ with a similar argument. In this case we use finitely many points $z_j$ where the line segments touch.
\end{proof}

If we denote by $\BFT_{n,L}$ the problem $\BFT_n$ restricted to functions that are Lip\-schitz continuous with constant $L$, 
then we can formulate our main result on Lip\-schitz continuous functions as follows (using Theorems~\ref{thm:BFT} and \ref{thm:Lipschitz}).

\begin{corollary}
$\BFT_{n,L}\equivSW\ConC_n$ for all $n\in\IN$ and $L>1$.
\end{corollary}

We mention that the problem $\BFT_{n,L}$ is obviously computable for $L<1$,
since fixed points of contractions are uniquely determined by the Banach Fixed Point Theorem.
The boundary case $L=1$ has been studied by Neumann~\cite[Theorem~5.8]{Neu15} in the context of more general versions of the Browder-G\"ohde-Kirk Fixed Point Theorem.
In this case $\BFT_{n,1}\equivW\XC_n$, where $\XC_n$ denotes convex choice for the space $[0,1]^n$, 
i.e., $\C_{[0,1]^n}$ restricted to convex sets. Convex choice was further studied by Le Roux and Pauly~\cite[Corollary~3.31]{LRP15a} and they proved
among other results that one actually obtains a strictly increasing chain of problems with increasing dimension, i.e., 
\[\ConC_1\equivW\XC_1\lW\XC_2\lW\XC_3\lW...\lW \C_{[0,1]}.\]
Hence, in general one has a trichotomy for the complexity of $\BFT_{n,L}$ in the cases $L<1, L=1$ and $L>1$.
In the one-dimensional case, one is left with a dichotomy since it follows from Neumann's result that 
\[\BFT_{1,1}\equivW\XC_1\equivW\ConC_1\equivW\BFT_1.\]

\section{Aspects of Dimension}
\label{sec:dimension}

In this section we want to discuss aspects of dimension of connected choice and the Brouwer Fixed Point Theorem.
Our main result is that connected choice is computably universal or complete from dimension three upwards 
in the sense that it is strongly equivalent to Weak K\H{o}nig's Lemma, which is one of the degrees of major importance.
In order to prove this result, we use the following geometric construction.

\begin{proposition}[Twisted cube]
\label{prop:twisted-cube}
The function 
\[T:\In\AA_-[0,1]\to\AA_3,A\mapsto(A\times[0,1]\times\{0\})\cup(A\times A\times[0,1])\cup([0,1]\times A\times\{1\})\]
is computable and maps non-empty closed sets $A\In[0,1]$
to non-empty pathwise connected closed sets $T(A)\In[0,1]^3$.
\end{proposition}

Here tuples $(x_1,x_2,x_3)\in T(A)$ have the property that at least one of the
first two components provide a solution $x_i\in A$, and the third component lets us
pick one that surely does. If $x_3$ is close to $1$, then
surely $x_2\in A$ and if $x_3$ is close to $0$, then surely $x_1\in A$. If $x_3$
is neither close to $0$ nor $1$, then both $x_1,x_2\in A$. 
Hence, there is a computable function $H$ such that $\C_{[0,1]}=H\circ\ConC_3\circ T$,
which proves $\C_{[0,1]}\leqSW\ConC_3$. Together with Theorem~\ref{thm:BFT} and Fact~\ref{fact:WKL}
we obtain the following conclusion.

\begin{theorem}[Completeness of three dimensions]
\label{thm:dimension-three}
For $n\geq 3$ we obtain
\[\ConC_n\equivSW\BFT_n\equivSW\BFT_\infty\equivSW\WKL\equivSW\C_{[0,1]}.\]
\end{theorem}
\begin{proof}
We note that the reduction $\ConC_n\leqSW\C_{[0,1]^n}$ holds for all $n\in\IN$,
since connected choice is a just a restriction of closed choice and the equivalences 
\[\C_{[0,1]^\IN}\equivSW\C_{[0,1]^n}\equivSW\C_{[0,1]}\equivSW\WKL\]
are known for all $n\geq1$ by Fact~\ref{fact:WKL}.
The equivalence $\ConC_n\equivSW\BFT_n$ has been proved in Theorem~\ref{thm:BFT}
for all $n\in\IN$. We mention that $\BFT_n\leqSW\BFT_\infty$ can be proved as follows.
The function 
\[K:\CC([0,1]^n,[0,1]^n)\to\CC([0,1]^\IN,[0,1]^\IN),f\mapsto((x_i)\mapsto(f(x_1,...,x_n),0,0,0,...))\]
is computable and together with the projection on the first $n$--coordinates this yields
the reduction $\BFT_n\leqSW\BFT_\infty$. Since 
\[\CC([0,1]^\IN,[0,1]^\IN)\to\AA_-([0,1]^\IN),f\mapsto(f-\id_{[0,1]^\IN})^{-1}\{0\}\]
is computable too, it follows that $\BFT_\infty\leqSW\C_{[0,1]^\IN}$ holds.
Finally, $\C_{[0,1]}\leqSW\ConC_n$ follows for $n\geq3$ from Proposition~\ref{prop:twisted-cube}.
\end{proof}

In particular, we get the Baigger counterexample for dimension $n\geq3$ as a consequence
of Theorem~\ref{thm:dimension-three}.
A superficial reading of the results of Orevkov~\cite{Ore63} and Baigger~\cite{Bai85} can lead to the wrong
conclusion that they actually provide a reduction of Weak K\H{o}nig's Lemma to the Brouwer Fixed Point Theorem
$\BFT_n$ of any dimension $n\geq2$. However, this is only correct in a non-uniform way and the corresponding
uniform result will be settled in Section~\ref{sec:dim2} with different methods and does not follow from the known constructions.
The Orevkov-Baigger result is built on the following fact.

\begin{proposition}[Mixed cube]
\label{prop:mixed-cube}
The function 
\[M:\In\AA_-[0,1]\to\AA_2,A\mapsto(A\times[0,1])\cup([0,1]\times A)\]
is computable and maps non-empty closed sets $A\In[0,1]$
to non-empty pathwise connected closed sets $M(A)\In[0,1]^2$.
\end{proposition}

It follows straightforwardly from the definition that the pairs $(x,y)\in M(A)$ are
such that one out of two components $x,y$ is actually in $A$. 
In order to express the uniform content of this fact, we introduce the concept of a fraction.

\begin{definition}[Fractions]
Let $f:\In X\mto Y$ be a multi-valued function and $0<n\leq m\in\IN$. We define the 
{\em fraction} $\frac{n}{m}f:\In X\mto Y^m$ by
\[\frac{n}{m}f(x):=\{(y_1,...,y_m)\in\range(f)^m:|\{i:y_i\in f(x)\}|\geq n\}\]
for all $x\in\dom(\frac{n}{m}f):=\dom(f)$.
\end{definition}

The idea of a fraction $\frac{n}{m}f$ is that it provides $m$ potential answers for $f$,
at least $n\leq m$ of which have to be correct. The uniform content of the Orevkov-Baigger construction is
then summarized in the following result.

\begin{proposition}[Connected choice in dimension two]
\label{prop:dimension-two}
$\frac{1}{2}\C_{[0,1]}\leqSW\ConC_2\leqSW\C_{[0,1]}$.
\end{proposition}
\begin{proof}
With Proposition~\ref{prop:mixed-cube} we obtain $\frac{1}{2}\C_{[0,1]}=\ConC_2\circ M$
and hence $\frac{1}{2}\C_{[0,1]}\leqSW\ConC_2$. The other reduction follows
from $\ConC_2\leqSW\C_{[0,1]^2}\equivSW\C_{[0,1]}$.
\end{proof}

That is, given a closed set $A\In[0,1]$ we can utilize connected choice $\ConC_2$ of dimension $2$
in order to find a pair of points $(x,y)$ one of which is in $A$. 
This result directly implies the counterexample of Baigger~\cite{Bai85} because
the fact that there are non-empty co-c.e.\ closed sets $A\In[0,1]$ without
computable point immediately implies that $\frac{1}{2}\C_{[0,1]}$ is not non-uniformly
computable (i.e., there are computable inputs without computable outputs) and hence
$\ConC_2$ is also not non-uniformly computable.

\begin{corollary}[Orevkov 1963, Baigger 1985]
\label{cor:Orevkov-Baigger}
There exists a computable function $f:[0,1]^2\to[0,1]^2$ that has no computable fixed point $x\in[0,1]^2$.
There exists a non-empty connected co-c.e.\ closed subset $A\In[0,1]^2$ without computable point.   
\end{corollary}

We mention that Proposition~\ref{prop:dimension-two} does not directly imply $\C_{[0,1]}\equivSW\ConC_2$,
since $\frac{1}{2}\C_{[0,1]}\lW\ConC_2$. In fact, we can prove an even stronger result which shows
that $\frac{1}{2}\C_{[0,1]}$ computes almost nothing, not even choice for the two point space.\footnote{Such problems 
have been called {\em indiscriminative} in \cite{BHK17a}.}
This means that Proposition~\ref{prop:dimension-two} has very little uniform content.

\begin{proposition}
\label{prop:half-choice}
$\C_{\{0,1\}}\nleqW\frac{1}{2}\C_{[0,1]}$.
\end{proposition}
\begin{proof}
We use $\C_{\{0,1\}}\equivSW\LLPO$ and 
by $\psi_-$ we denote the representation of $\AA_1$. 
We recall that $\LLPO:\In\IN^\IN\mto\IN$ is defined such that for $j\in\{0,1\}$ and $p\in\{0,1\}^\IN$ it holds that
\[j\in\LLPO(p)\iff(\forall i)\;p(2i+j)=0,\]
where $\dom(\LLPO)$ contains all sequences $p$ such that $p(k)\not=0$ for at most one $k$.

Let us now assume that $\LLPO\leqW\frac{1}{2}\C_{[0,1]}$ holds.
Then there are continuous $H,K$ such that $H\langle\id,FK\rangle$ realizes $\LLPO$
whenever $F$ realizes $\frac{1}{2}\C_{[0,1]}$.
We consider the inputs $p_{ji}:=0^{2i+j+1}10^\IN$ and $p_\infty:=0^\IN$ for $\LLPO$.
We obtain $\LLPO(p_{ji})=\{j\}$ for $j\in\{0,1\}$ and $\LLPO(p_\infty)=\{0,1\}$.
Now we let $K_{ji}:=\psi_-(K(p_{ji}))$ and $K_\infty:=\psi_-(K(p_\infty))$. These sets
are all non-empty compact subsets of $[0,1]$, hence there are $x_{ji}\in K_{ji}$ with names $q_{ji}$ 
(with respect to the signed-digit representation of $[0,1]$).
Due to compactness for each $j\in\{0,1\}$ there is some convergent subsequence $(q_{ji_k})$ of $q_{ji}$
and we let $q_j:=\lim_{k\to\infty}q_{ji_k}$ and $x_j:=\lim_{k\to\infty}x_{ji_k}$.

Now we claim that $x_j\in K_\infty$ for both $j\in\{0,1\}$ and by symmetry it suffices to prove
this for $j=0$. Let us assume that $x_0\not\in K_\infty$. Then by continuity of $K$ there
exists some open neighborhood $U$ of $x_0$ and some $k\in\IN$ such that
$U\cap\psi_-(K(r))=\emptyset$ for all $r\in\dom(\psi_-K)$ with $0^k\prefix r$.
Almost all $p_{0i}$ satisfy this condition, which implies $U\cap K_{0i}=\emptyset$ for almost all $i$.
This contradicts the construction of $x_0$. Hence $x_0\in K_\infty$ follows and analogously $x_1\in K_\infty$.

Hence there is some realizer $F_\infty$ of $\frac{1}{2}\C_{[0,1]}$ with $F_\infty K(p_\infty)=\langle q_0,q_1\rangle$.
Without loss of generality we can assume $H\langle p_\infty,\langle q_0,q_1\rangle\rangle=0$.
There are also realizers $F_k$ of $\frac{1}{2}\C_{[0,1]}$ with 
$F_kK(p_{1i_k})=\langle q_{0i_k},q_{1i_k}\rangle$, since the second component contains a correct answer.
Hence $H\langle p_{1i_k},\langle q_{0i_k},q_{1i_k}\rangle\rangle=1$ has to hold.
Continuity of $H$ now implies $H\langle p_\infty,\langle q_0,q_1\rangle\rangle=1$,
which is a contradiction.
\end{proof}

In the following result we summarize the known relations for connected choice in dependency of the dimension.

\begin{proposition}
\label{prop:connected-dimension}
We obtain
$\ConC_0\lW\ConC_1\lW\ConC_2\leqW\ConC_n\equivW\C_{[0,1]}$ for all $n\geq3$.
\end{proposition}
\begin{proof}
It is clear that
$\ConC_n\leqSW\ConC_{n+1}$ holds for all $n\in\IN$, since the computable map $A\mapsto A\times[0,1]$
maps connected closed sets of dimension $n$ to such sets of dimension $n+1$.
The reduction $\ConC_0\lW\ConC_1$ is strict, since $\ConC_0$ is computable and $\ConC_1$ is not.
The reduction $\ConC_1\lW\ConC_2$ is strict, since $\ConC_1$ is non-uniformly computable (since any non-empty
connected co-c.e.\ closed set $A\In[0,1]$ is either a singleton and hence computable or it has non-empty interior and contains even a rational point)
and $\ConC_2$ is not non-uniformly computable by Corollary~\ref{cor:Orevkov-Baigger}.
\end{proof}

In Section~\ref{sec:dim2} we are going to prove that also $\ConC_2\equivSW\C_{[0,1]}$ holds. 

We close this section with a second proof of Theorem~\ref{thm:dimension-three} that uses a combinatorial argument 
as a replacement for the geometric construction provided in Proposition~\ref{prop:twisted-cube}.
It also indicates special properties of dimension two, which are not shared by higher dimensions.
Firstly, one can extend Proposition~\ref{prop:mixed-cube} straightforwardly to higher dimensions
(by choosing $A\mapsto(A\times A\times[0,1])\cup(A\times [0,1]\times A)\cup([0,1]\times A\times A)$
in dimension three and so forth) and that leads to the following generalization of Proposition~\ref{prop:dimension-two}.

\begin{proposition} $\frac{n-1}{n}\C_{[0,1]}\leqSW\ConC_n\leqSW\C_{[0,1]}$ for all $n\geq2$.
\end{proposition}

On the other hand, one can use a majority voting strategy to obtain the following result.

\begin{proposition}[Majority vote] 
\label{prop:majority}
$\frac{k}{n}\C_{[0,1]} \equivSW\C_{[0,1]}$ if $2k>n\geq k>0$.
\end{proposition} 
\begin{proof} 
It is clear that $\frac{k}{n}\C_{[0,1]} \leqSW\C_{[0,1]}\equivSW\widehat{\LLPO}$ holds.
Hence, we only need to prove $\widehat{\LLPO}\leqSW\frac{k}{n}\C_{[0,1]}$.
In the first step we show $\widehat{\LLPO} \leqSW \frac{k}{n}\widehat{\LLPO}$. 
Given some answer $(p_1, \ldots, p_{n}) \in \frac{k}{n}\widehat{\LLPO}(q)$, 
a solution $p \in \widehat{\LLPO}(q)$ can be obtained by bitwise majority voting: 
for any $i \in\IN$, we let $p(i) := 1$ if and only if $|\{j : p_j(i) = 1\}| \geq k$ 
and $p(i) := 0$ otherwise. This guarantees majority since $2k>n$.
To complete the proof it suffices to show $\frac{k}{n}\widehat{\LLPO}\leqSW\frac{k}{n}\C_{[0,1]}$.
We know that $\widehat{\LLPO}\leqSW\C_{[0,1]}$ and hence there are computable $H,K$ such that
$HFK\vdash\widehat{\LLPO}$ whenever $F\vdash\C_{[0,1]}$ holds. Without loss of generality,
we can assume that we use a total representation for $[0,1]$ and hence $H$ has to be total since $\C_{[0,1]}$ is surjective.
This implies that $H^{n}FK\vdash\frac{k}{n}\widehat{\LLPO}$ whenever $F\vdash\frac{k}{n}\C_{[0,1]}$,
which completes the proof.
\end{proof} 

We note that $\frac{n-1}{n}$ satisfies $2(n-1)>n$ if and only if $n\geq3$.
This does constitute a second proof of Theorem~\ref{thm:dimension-three}.
Moreover, Proposition~\ref{prop:half-choice} shows that the claim of Proposition~\ref{prop:majority} does not hold for $n=2$ and $k=1$.
This illustrates from a combinatorial perspective why dimension two is special.

\section{The Two-Dimensional Case}
\label{sec:dim2}

The goal of this section is to prove that connected choice $\ConC_n$ is equivalent to $\C_{[0,1]}$ 
in the two-dimensional case $n=2$. The construction required for the proof of $\C_{[0,1]}\leqSW\ConC_2$ 
is much more involved than in the three-dimensional
case and it is essentially based on an inverse limit construction. 

\begin{theorem}[Two-dimensional case]
\label{thm:dim2}
$\ConC_2\equivSW\BFT_2\equivSW\C_{[0,1]}$.
\end{theorem}
\begin{proof}
By Theorem~\ref{thm:BFT} and Fact~\ref{fact:WKL} it is sufficient to show $\widehat{\LLPO}\leqSW\ConC_2$.
In order to make the proof more understandable, we structure it into several parts.\\

\noindent
{\bf Preparation of the input.}
In order to organize the input information, we replace $\widehat{\LLPO}$ by an equivalent problem $\LLPO_\infty$
that we now define.
In the following we denote pairs $(n,b)\in\IN\times\{0,1\}$ for simplicity by $n_b$.
We say that a word $w\in(\IN\times\{0,1\})^*$ is {\em repetition-free}, if 
no number appears twice in the first component, i.e., if $w={n_0}_{b_0}{n_1}_{b_1}...{n_k}_{b_k}$, then 
$n_i\not=n_j$ for all $i,j\leq k$ with $i\not=j$. We introduce the following sets of repetition-free words.
 
\begin{localdef}[Repetition-free words]
For all $n\in\IN$ we define the sets
\begin{enumerate}
\item $W_n:=\{w\in(\{0,...,n-1\}\times\{0,1\})^*:w$ repetition-free$\}$,
\item $W_*:=\bigcup_{n\in\IN}W_n$,
\item $W_\IN:=\{p\in(\IN\times\{0,1\})^\IN:(\forall k)\;p|_k\in W_*\}$.
\item $W_\infty:=W_*\cup W_\IN$.
\end{enumerate}
\end{localdef}

For instance, $W_2=W_1\cup\{1_0,1_1,0_01_0,0_01_1,0_11_0,0_11_1,1_00_0,1_00_1,1_10_0,1_10_1\}$,
where $W_1=W_0\cup\{0_0,0_1\}$ and $W_0=\{\varepsilon\}$. 

We note that we use a representation $\delta_{W_\infty}$ to represent $\W_\infty$ that enumerates
the content of words. More precisely, we consider $p=\langle n_0,b_0\rangle\langle n_1,b_1\rangle\langle n_2,b_2\rangle...$
with $n_i\in\IN, b_i\in\{0,1\}$ as a name of a sequence ${n_0}_{b_0}{n_1}_{b_1}{n_2}_{b_2}...$
in which we remove all the ${n_i}_{b_i}$ with $n_i=0$ or with an $n_i$ that occurred already earlier in the sequence
and then we replace all the remaining ${n_i}_{b_i}$ by ${(n_i-1)}_{b_i}$. The resulting object is a finite
or infinite sequence $q\in W_\infty$ and we set $\delta_{W_\infty}(p)=q$.

Now we can define the problem $\LLPO_\infty$.

\begin{localdef}
$\LLPO_\infty:W_\infty\mto\{0,1\}^\IN$ is defined by
\[\LLPO_\infty(p):=\{q\in\{0,1\}^\IN:(\forall n,b)(q(n)=b\TO n_b\not\in\range(p))\}\]
for all $p\in W_\infty$.
\end{localdef}

\begin{localclaim}
\label{claim:dim2-1}
$\widehat{\LLPO}\equivSW\LLPO_\infty$.
\end{localclaim}
\begin{proof}
``$\widehat{\LLPO}\leqSW\LLPO_\infty$'': Given an input $p=\langle p_0,p_1,...\rangle$ for $\widehat{\LLPO}$ we generate
a repetition-free sequence $q\in W_\infty$ as follows. As soon as we learn from $p_n$ that $b\not\in\LLPO(p_n)$, then we write $n_b$ into the output.
Hence, $n_b$ occurs in the output sequence if and only if the $n$--th copy of $\LLPO$ does not allow the result $b$.
Hence, the resulting sequence $q\in W_\infty$ satisfies $\LLPO_\infty(q)=\widehat{\LLPO}(p)$. 
``$\LLPO_\infty\leqSW\widehat{\LLPO}$'': Vice versa, given a sequence $q\in W_\infty$, we can generate
a suitable input $p=\langle p_0,p_1,...\rangle$ to $\LLPO$ as follows. We start all $p_i$ with zeros and as soon
as we read some $n_b$ in $q$ we modify $p_n$ so that it contains a $1$ in some still available position, i.e., we set $p_n(2k+b)=1$ for large enough $k$. All other positions of $p_n$ will be filled with $0$.
In this way we obtain a sequence $p$ with $\LLPO_\infty(q)=\widehat{\LLPO}(p)$.  
\end{proof}

Now our goal is now to prove $\LLPO_\infty\leqSW\ConC_2$.
In fact, for convenience we replace $[0,1]^2$ by $B_0:=[0,1]\times [0,3]$ and we show
$\LLPO_\infty\leqSW\ConC_{B_0}$, where $\ConC_{B_0}$ is connected choice for the space $B_0$.
It is clear that $\ConC_{B_0}\equivSW\ConC_2$.
\\

\noindent
{\bf Overview of the proof.}
Given a repetition-free sequence $p\in W_\infty$, i.e., an input to $\LLPO_\infty$, we will compute a connected non-empty set
\[A(p):=\{x\in B_0:(\forall n\in\IN)\; f_{n-1}^{-1}\circ...\circ f_0^{-1}(x)\in E_n(s_n(p))\}\In\IR^2\]
that is defined by an inverse limit construction. That means that the functions $f_n:B_{n+1}\into B_n$ are computable
embeddings of certain rectangles $B_n\In\IR^2$ (called {\em blocks}) into each other and $E_n(s_n(p))\In B_n$ are certain subsets that consists
of a union of finitely many squares (called {\em tiles}) within $B_n$. These sets $E_n$ will be constructed such that
they reflect the information encoded in a certain portion $s_n(p)\in W_{n+1}$ of $p$ and this encoding will be organized such that
any point $y\in A(p)$ will allow us to compute some possible value of $\LLPO_\infty(p)$. 
We first describe the construction of the discrete structure of these blocks $B_n$ and certain subsets $S_n\In B_n$ (called {\em snakes})
that are completely independent of the input $p$. In a second step we describe how the sets $E_n$ are constructed as subsets of the snakes $S_n$ 
in dependence of the input $p$. Then we define the computable embeddings $f_n$ such that they preserve the information encoded in
the sets $E_n$ in a particular way. In the next step we show that the sets $A(p)$ can be computed from $p$ and that they allow us
to recover the information $\LLPO_\infty(p)$ from any $y\in A(p)$. Finally, we show that the sets $A(p)$ are non-empty and connected. \\

\begin{figure}[tb]
\begin{scriptsize}
\hspace*{-1cm}\begin{tikzpicture}[xscale=0.6,yscale=0.6]
\useasboundingbox  rectangle (13.6653,12.9987);

% S_0

\draw[fill=red!20, draw=red!50!black]  (5,3) rectangle (6,2);
\draw[fill=black!20]  (5,4) rectangle (6,3);
\draw[fill=green!20, draw=green!50!black]  (5,5) node (v23) {} rectangle (6,4);

\node at (5.5,4.5) {$0_1$};
\node at (5.5,2.5) {$0_0$};

\node at (5.5,1) {\normalsize $S_0$};

%  S_1

\draw[step=1cm,gray,very thin] (10,2) grid (13,11);

\draw[fill=red!20, draw=red!50!black]  (10,3) rectangle (11,2);
\draw[fill=black!20]  (10,4) rectangle (11,3);
\draw[fill=green!20, draw=green!50!black]  (10,5) rectangle (11,4);

\draw[fill=red!20!white, draw=red!50!black]  (12,5) rectangle (13,4);
\draw[fill=white!80!black]  (12,6) rectangle (13,5);
\draw[fill=green!20!white, draw=green!50!black]  (12,7) rectangle (13,6);

\draw[fill=white!80!black]  (11,5) rectangle (12,4) node (v2) {};

\draw[fill=white!80!black]  (12,8) rectangle (13,7);
\draw[fill=red!20!white, draw=red!50!black]  (10,9) rectangle (13,8);
\draw[fill=white!80!black]  (10,10) rectangle (11,9);
\draw[fill=green!20!white, draw=green!50!black]  (10,11) rectangle (13,10);

\node at (11.5,10.5) {$1_1$};
\node at (11.5,8.5) {$1_0$};
\node at (12.5,6.5) {$0_11_1$};
\node at (12.5,4.5) {$0_11_0$};
\node at (10.5,4.5) {$0_01_1$};
\node at (10.5,2.5) {$0_01_0$};

\node at (11.5,1) {\normalsize $S_1$};

% c_1

\node (v1) at (13.5,2.5) {\normalsize $c_1$};
\draw [->] (v1) edge (v2);

% f_0

\node (v3) at (10.5,11) {};
\node (v7) at (11.5,11) {};
\node (v12) at (12.5,11) {};
\node (v16) at (6,4.5) {};
\node (v11) at (6,3.5) {};
\node (v17) at (6,2.5) {};
\node (v4) at (10.5,11.5) {};
\node (v8) at (11.5,12) {};
\node (v13) at (12.5,12.5) {};
\node (v5) at (9,11.5) {};
\node (v9) at (8.5,12) {};
\node (v14) at (8,12.5) {};
\node (v15) at (8,4.5) {};
\node (v10) at (8.5,3.5) {};
\node (v6) at (9,2.5) {};

\draw [-](v3) -- (v4.center) -- (v5.center) -- (v6.center);
\draw [-](v7) -- (v8.center) -- (v9.center) -- (v10.center);
\draw [-](v12) -- (v13.center) -- (v14.center) -- (v15.center);
\draw [-] (v15.center) edge (v16);
\draw [->] (v15) edge (v16);
\draw [-] (v10.center) edge (v11);
\draw [->] (v10) edge (v11);
\draw [-] (v6.center) edge (v17);
\draw [->] (v6) edge (v17);

\node at (8.5,1) {\normalsize $f_0$};

% f_0(S_1)

\draw (0.3334,10.9989) node (v18) {} rectangle (3.3334,7.9989);
\draw (0.3334,7.9989) rectangle (3.3334,4.9989);
\draw (0.3334,4.9989) rectangle (3.3334,1.9989) node (v22) {};

\draw [fill=green!20!white, draw=green!50!black]  (v18) rectangle (0.6654,1.9987);
\draw [fill=white!80!black] (0.6654,4.9984) rectangle (0.9987,1.9987) node (v19) {};
\draw [fill=red!20, draw=red!50!black]  (v19) rectangle (1.332,10.9978) node (v20) {};
\draw [fill=white!80!black] (v20) rectangle (1.6653,7.9981);
\draw [fill=green!20!white, draw=green!50!black]  (1.6653,10.9978) rectangle (1.9986,7.9981);
\draw [fill=white!80!black] (1.9986,10.9978) rectangle (2.3319,7.9981);
\draw [fill=red!20, draw=red!50!black]  (2.3319,10.9978) rectangle (2.6652,7.9981);
\draw [fill=white!80!black] (2.3319,7.9981) rectangle (2.6652,4.9984) node (v21) {};
\draw [fill=green!20!white, draw=green!50!black]  (2.3319,4.9984) rectangle (2.6652,1.9987);
\draw [fill=white!80!black] (v21) rectangle (2.9985,1.9987);
\draw [fill=red!20, draw=red!50!black]  (2.9985,4.9984) rectangle (v22);

\node at (2,1) {\normalsize $f_0(S_1)$};

\node (v24) at (3.333,10.9989) {};
\node (v25) at (4.9995,1.9998) {};
\node (v26) at (5.9994,4.9995) {};
\draw [-,dotted] (v23) edge (v24);
\draw [-,dotted] (v22) edge (v25);
\node at (-1.3332,-4.3329) {};
\end{tikzpicture}
\end{scriptsize}
\ \\[-0.6cm]
\caption{The embedding $f_0:B_1\into S_0$.}
\label{fig:f0}
\end{figure}

\noindent
{\bf The discrete structure of blocks and snakes within them.}
We will now describe a discrete structure within $\IR^2$ that will be used to represent information from repetition-free words. 
This structure consists of certain {\em blocks} $B_n := [0, w_n] \times [0,h_n]$
of a suitable {\em width} $w_n$ and and a suitable {\em height} $h_n$.
We call subsets of the form $[i,i+1]\times[j,j+1]\In\IR^2$ {\em tiles}.
Within the blocks $B_n$ we identify subsets $G_n,M_n,R_n\In B_n$ that 
are unions of tiles. The sets $G_n$ will be displayed in green and they will be used to encode certain bits of value $1$.
The sets $R_n$ will be displayed in red and they will be used to encode certain bits of value $0$.
The sets $M_n$ will be displayed in gray and they are middle sets that are used to separate bits. 
The union $S_n:=G_n\cup M_n\cup R_n$ will constitute a connected chain of tiles, called {\em snake}.
The construction proceeds inductively using previous parts
of the same color that are shifted using certain {\em corner points} $c_n$. 
For a point $c\in\IR^2$ and a set $A\In\IR^2$, we use the notation
$c+A:=\{c+x:x\in A\}$. 

\begin{localdef}[Blocks]
We define numbers $w_n,h_n\in\IN$, points $c_n\in\IR^2$ and
sets $B_n,G_n, M_n, R_n,S_n \subseteq\IR^2$ for all $n\in\IN$ as follows:
\begin{enumerate}
\item $w_n := 2n(w_{n-1} + 1) - 1$ for $n>0$, $w_0 := 1$ \hfill (width)
\item $h_{n} := 2n(h_{n-1}-1) + 5$  for $n>0$, $h_0 := 3$ \hfill (height)
\item $c_n := (w_{n-1}+1,h_{n-1}-1)$ for $n>0$ \hfill (corner)
\item $B_n := [0, w_n] \times [0,h_n]$ \hfill (block)
\item $G_{n} := [0,w_n] \times [h_n-1,h_n] \cup \bigcup_{k=0}^{2n-1} (kc_n+G_{n-1} )$  \hfill (green)
\item $R_{n} := [0,w_n] \times [h_n-3,h_n-2] \cup \bigcup_{k=0}^{2n-1} (kc_n + R_{n-1})$  \hfill (red)
\item $M_{n}  := ([0,1] \times [h_n-2, h_n-1]) \cup ([w_n-1,w_n] \times [h_n-4, h_n-3])\cup \bigcup_{k=0}^{2n-1}  \left( kc_n + M_{n-1} \right) \cup  \bigcup_{k=1}^{2n-1} \left(kc_n + ([- 1,0] \times [0,1]) \right)$ \hfill (middle)
\item $S_n := G_n \cup R_n \cup M_n$ \hfill (snake)
\end{enumerate}
\end{localdef}

The construction is illustrated in Figures \ref{fig:f0} and \ref{fig:f1}. The sets $G_n$, $M_n$, and $R_n$ are green, gray, and red, respectively.
The following observations are immediate.

\begin{localclaim}
\label{claim:dim2-2}
$G_n,M_n,R_n,S_n\In\IR^2$ are closed and they satisfy for all $n\in\IN$:
\begin{enumerate}
\item $S_n=G_n\cup M_n\cup R_n\In B_n$ and $G_n\cap R_n=\emptyset$,
\item $S_n$ is a chain of tiles, i.e., all tiles in $S_n$ but $[0,1] \times [0,1]$ and $[w_n-1,w_n] \times [h_n-1,h_n]$ are edge-connected to exactly two other tiles in $S_n$,
\item $kc_n + S_{n-1} \subseteq S_n$ for all $0 \leq k \leq 2n$.
\end{enumerate}
\end{localclaim}

\begin{figure}[htb]
\begin{tiny}
\hspace*{-1cm}\begin{tikzpicture}[xscale=0.82,yscale=0.3]
\useasboundingbox  rectangle (17.5,48.5);

% Snake S_2

\draw[step=1cm,gray,very thin] (2,0) grid (17,37);

\draw[fill=red!20, draw=red!50!black]  (2,1) rectangle (3,0);
\draw[fill=black!20]  (2,2) rectangle (3,1);
\draw[fill=green!20, draw=green!50!black]  (2,3) rectangle (3,2);
\draw[fill=red!20!white, draw=red!50!black]  (4,3) rectangle (5,2);
\draw[fill=white!80!black]  (4,4) rectangle (5,3);
\draw[fill=green!20!white, draw=green!50!black]  (4,5) rectangle (5,4);
\draw[fill=white!80!black]  (3,3) rectangle (4,2);
\draw[fill=white!80!black]  (4,6) rectangle (5,5);
\draw[fill=red!20!white, draw=red!50!black]  (2,7) rectangle (5,6);
\draw[fill=white!80!black]  (2,8) rectangle (3,7);
\draw[fill=green!20!white, draw=green!50!black]  (2,9) rectangle (5,8);

\node at (3.5,8.5) {$0_02_1$};
\node at (3.5,6.5) {$0_02_0$};
\node at (4.5,4.5) {$0_01_12_1$};
\node at (4.5,2.5) {$0_01_12_0$};
\node at (2.5,2.5) {$0_01_02_1$};
\node at (2.5,0.5) {$0_01_02_0$};

\draw[fill=white!80!black]  (5,9) rectangle (6,8);

\draw[step=1cm,gray,very thin] (6,8) node (v3) {} grid (9,17) node (v1) {};
\draw[fill=red!20, draw=red!50!black]  (6,9) rectangle (7,8);
\draw[fill=black!20]  (6,10) rectangle (7,9);
\draw[fill=green!20, draw=green!50!black]  (6,11) rectangle (7,10);
\draw[fill=red!20!white, draw=red!50!black]  (8,11) rectangle (9,10);
\draw[fill=white!80!black]  (8,12) rectangle (9,11);
\draw[fill=green!20!white, draw=green!50!black]  (8,13) rectangle (9,12);
\draw[fill=white!80!black]  (7,11) rectangle (8,10);
\draw[fill=white!80!black]  (8,14) rectangle (9,13);
\draw[fill=red!20!white, draw=red!50!black]  (6,15) rectangle (9,14);
\draw[fill=white!80!black]  (6,16) rectangle (7,15);
\draw[fill=green!20!white, draw=green!50!black]  (6,17) rectangle (9,16);

\node at (7.5,16.5) {$0_12_1$};
\node at (7.5,14.5) {$0_12_0$};
\node at (8.5,12.5) {$0_11_12_1$};
\node at (8.5,10.5) {$0_11_12_0$};
\node at (6.5,10.5) {$0_11_02_1$};
\node at (6.5,8.5) {$0_11_02_0$};

\draw[fill=white!80!black]  (9,17) rectangle (10,16);

\draw[step=1cm,gray,very thin] (10,16) grid (13,25) node (v1) {};
\draw[fill=red!20, draw=red!50!black]  (10,17) rectangle (11,16);
\draw[fill=black!20]  (10,18) rectangle (11,17);
\draw[fill=green!20, draw=green!50!black]  (10,19) rectangle (11,18);
\draw[fill=red!20!white, draw=red!50!black]  (12,19) rectangle (13,18);
\draw[fill=white!80!black]  (12,20) rectangle (13,19);
\draw[fill=green!20!white, draw=green!50!black]  (12,21) rectangle (13,20);
\draw[fill=white!80!black]  (11,19) rectangle (12,18);
\draw[fill=white!80!black]  (12,22) rectangle (13,21);
\draw[fill=red!20!white, draw=red!50!black]  (10,23) rectangle (13,22);
\draw[fill=white!80!black]  (10,24) rectangle (11,23);
\draw[fill=green!20!white, draw=green!50!black]  (10,25) rectangle (13,24);

\node at (11.5,24.5) {$1_02_1$};
\node at (11.5,22.5) {$1_02_0$};
\node at (12.5,20.5) {$1_00_12_1$};
\node at (12.5,18.5) {$1_00_12_0$};
\node at (10.5,18.5) {$1_00_02_1$};
\node at (10.5,16.5) {$1_00_02_0$};

\draw[fill=white!80!black]  (13,25) rectangle (14,24);

\draw[step=1cm,gray,very thin] (14,24) grid (17,33) node (v1) {};
\draw[fill=red!20, draw=red!50!black]  (14,25) rectangle (15,24);
\draw[fill=black!20]  (14,26) rectangle (15,25);
\draw[fill=green!20, draw=green!50!black]  (14,27) rectangle (15,26);
\draw[fill=red!20!white, draw=red!50!black]  (16,27) rectangle (17,26);
\draw[fill=white!80!black]  (16,28) rectangle (17,27);
\draw[fill=green!20!white, draw=green!50!black]  (16,29) rectangle (17,28);
\draw[fill=white!80!black]  (15,27) rectangle (16,26);
\draw[fill=white!80!black]  (16,30) rectangle (17,29);
\draw[fill=red!20!white, draw=red!50!black]  (14,31) rectangle (17,30);
\draw[fill=white!80!black]  (14,32) rectangle (15,31);
\draw[fill=green!20!white, draw=green!50!black]  (14,33) rectangle (17,32);

\node at (15.5,32.5) {$1_12_1$};
\node at (15.5,30.5) {$1_12_0$};
\node at (16.5,28.5) {$1_10_12_1$};
\node at (16.5,26.5) {$1_10_12_0$};
\node at (14.5,26.5) {$1_10_02_1$};
\node at (14.5,24.5) {$1_10_02_0$};

\draw[fill=white!80!black]   (16,34) rectangle (17,33);

\draw[fill=red!20!white, draw=red!50!black]  (2,35) rectangle (17,34);
\draw[fill=white!80!black]   (2,36) rectangle (3,35);
\draw[fill=green!20!white, draw=green!50!black]  (2,37) rectangle (17,36);

\node at (9.5,34.5) {$2_0$};
\node at (9.5,36.5) {$2_1$};

\node at (1,0.5) {\normalsize $S_2$};

% c_2

\node (v2) at (7.5,5.5) {\normalsize $c_2$};
\draw [->] (v2) edge (v3);

% S_1

\draw[fill=red!20, draw=red!50!black]  (2,40.5) rectangle (3,39.5);
\draw[fill=black!20]  (2,41.5) rectangle (3,40.5);
\draw[fill=green!20, draw=green!50!black]  (2,42.5) rectangle (3,41.5);
\draw[fill=red!20!white, draw=red!50!black]  (4,42.5) rectangle (5,41.5);
\draw[fill=white!80!black]  (4,43.5) rectangle (5,42.5);
\draw[fill=green!20!white, draw=green!50!black]  (4,44.5) rectangle (5,43.5);
\draw[fill=white!80!black]  (3,42.5) rectangle (4,41.5);
\draw[fill=white!80!black]  (4,45.5) rectangle (5,44.5);
\draw[fill=red!20!white, draw=red!50!black]  (2,46.5) rectangle (5,45.5);
\draw[fill=white!80!black]  (2,47.5) rectangle (3,46.5);
\draw[fill=green!20!white, draw=green!50!black]  (2,48.5) rectangle (5,47.5);

\node at (1,40) {\normalsize $S_1$};

% f_1

\node (v4) at (2.5,37) {};
\node (v6) at (3.5,37) {};
\node (v8) at (4.5,37) {};
\node (v10) at (5.5,37) {};
\node (v12) at (6.5,37) {};
\node (v24) at (7.5,37) {};
\node (v26) at (8.5,37) {};
\node at (10.5,37) {};
\node at (11.5,37) {};
\node (v22) at (12.5,37) {};
\node (v20) at (13.5,37) {};
\node (v18) at (14.5,37) {};
\node (v16) at (15.5,37) {};
\node (v14) at (16.5,37) {};
\node (v5) at (2.5,40) {$0_01_0$};
\node (v7) at (2.5,41) {};
\node (v9) at (2.5,42) {$0_01_1$};
\node (v11) at (3.5,42) {};
\node (v13) at (4.5,42) {$0_11_0$};
\node (v25) at (4.5,43) {};
\node (v27) at (4.5,44) {$0_11_1$};
\node at (4.5,45) {};
\node at (4.5,46) {};
\node at (3.5,46) {$1_0$};
\node (v23) at (2.5,46) {};
\node (v21) at (2.5,47) {};
\node (v19) at (2.5,48) {};
\node (v17) at (3.5,48) {$1_1$};
\node (v15) at (4.5,48) {};
\draw [->] (v4) edge (v5);
\draw [->] (v6) edge (v7);
\draw [->] (v8) edge (v9);
\draw [->] (v10) edge (v11);
\draw [->] (v12) edge (v13);
\draw [->] (v14) edge (v15);
\draw [->] (v16) edge (v17);
\draw [->] (v18) edge (v19);
\draw [->] (v20) edge (v21);
\draw [->] (v22) edge (v23);
\draw [->] (v24) edge (v25);
\draw [->] (v26) edge (v27);
\node at (9,38.5) {\normalsize $...$};

\node at (10.5,43.5) {\normalsize $f_1$};
\end{tikzpicture}
\end{tiny}
\ \\[-0.6cm]
\caption{The embedding $f_1:B_2\into S_1$.}
\label{fig:f1}
\end{figure}

\noindent
{\bf Coding of repetition-free words in the discrete structure.}
In this part of the proof we describe how the discrete structure with the 
snakes $S_n$ can be used to encode repetition-free words into sets $E_n\In S_n$ that depend on these words. 
For this purpose we first define a map that removes the number $k$ from
a repetition-free word and decrements all entries that are greater than $k$ by one.

\begin{localdef}
We define $r : \mathbb{N} \times W_* \to W_*$ as follows
for $u,v\in W_*,k\in\IN$:
\begin{enumerate}
\item $r(k,u) := u$ if $u \in W_{k}$,
\item $r(k,uk_bv) := r(k,uv)$,
\item $r(k,u(n+1)_bv) := r(k,u)n_br(k,v)$ for $n+1>k$.
\end{enumerate}
\end{localdef}

We will use this map $r$ in order to define a map $E_n:W_{n+1}\to\AA_-(S_n)$ that shows
how we encode repetition-free words $w\in W_{n+1}$ as closed subsets of $S_n$.

\begin{localdef}
\label{def:En}
For all $n\in\IN$ we define a map $E_n:W_{n+1}\to\AA_-(S_n)$ inductively as follows for $w\in W_{n+1}$: 
\begin{enumerate}
\item $E_n(\varepsilon) := S_n$
\item $E_n(n_0w) := [0,w_n] \times [h_n-3,h_n-2]$
\item $E_n(n_1w) := [0,w_n] \times [h_n-1,h_n]$
\item $E_n(k_bw) := (2k+b)c_n + E_{n-1}(r(k,w))$ for $k < n$
\end{enumerate}
\end{localdef}

The sets $E_0(w), E_1(w)$ and $E_2(w)$ are illustrated in Figures~\ref{fig:f0} and \ref{fig:f1} by the given words $w$.
By comparing the recursive definition of $E_n$ with the definitions of $G_n$ and $R_n$, we see that if $n_0 \in\range(w)$, 
then $E_n(w) \subseteq G_n$ and if $n_1 \in \range(w)$, then $E_n(w) \subseteq R_n$. 
Together with the fact that $G_n\cap R_n=\emptyset$ this will enable us to recover the bits $b$ with $n_b\in\range(w)$
from $E_n(w)$. 

\begin{localclaim}
\label{claim:dim2-3}
For all $n,k\in\IN$, $b\in\{0,1\}$ and $w\in W_{n+1}$  we have
\begin{enumerate}
\item $E_n(w) \subseteq S_n$,
\item $E_n(wk_b) \subseteq E_n(w)$,
\item $n_0 \in \range(w)\TO E_n(w) \subseteq R_n$, and $n_1 \in \range(w)\TO E_n(w) \subseteq G_n$.
\end{enumerate}
\end{localclaim}
\begin{proof}
We prove all claims by induction.
\begin{enumerate}
\item By induction on $w$, which holds trivially for $w = \varepsilon$. For the inductive case: $E_n(n_bw) \subseteq S_n$ by definition, and for $k < n$ we have 
\[\qquad\qquad E_n(k_bw) = (2k+b)c_n + E_{n-1}(r(k,w)) \subseteq (2k+b)c_n + S_{n-1}\] 
by induction hypothesis, so $E_n(k_bw) \subseteq S_n$ by Claim~\ref{claim:dim2-2}.

\item By induction on $w$. Base case, 
\[\qquad\qquad E_n(k_b) = {(2k+b)c_n + E_{n-1}(\varepsilon)} =  {(2k+b)c_n} + S_{n-1} \subseteq S_n = E_n(\varepsilon)\]  by Claim~\ref{claim:dim2-2}. 
Inductive case, $w = i_du$. First subcase, $i = n$: $E_n(n_duk_b) = E_n(n_du)$. Second subcase, $i < n$, so 
\[\qquad\qquad E_n(i_duk_b) = (2i+d)c_n + E_{n-1}(r(i,uk_b)) = (2i+d)c_n + E_{n-1}(r(i,u)k'_b)\] 
with $k' \in \{k,k-1\}$. By induction hypothesis $E_{n-1}(r(i,u)k'_b) \subseteq E_{n-1}(r(i,u))$ so $E_n(i_duk_b) \subseteq (2i+d)c_n + E_{n-1}(r(i,u)) = E_n(i_du)$.

\item We prove only the first statement, and by (2) it suffices to prove that $E_n(un_0) \subseteq R_n$ for all $u\in W_n$. 
By induction on $u$. In the base case, $E_n(\varepsilon n_0) \subseteq R_n$ holds by definition. In the inductive case we obtain
\[\qquad\qquad E_n(k_bun_0) = (2k+b)c_n + E_{n-1}(r(k,u)(n-1)_0) \subseteq (2k+b)c_n + R_{n-1}\]
by induction hypothesis. Since $(2k+b)c_n + R_{n-1} \subseteq R_n$ by definition, it follows that $E_n(k_bun_0) \subseteq R_n$.
\qedhere
\end{enumerate}
\end{proof}

Finally, we mention that the recursive definition of $E_n$ together with Claim~\ref{claim:dim2-3}~(2) implies the following.

\begin{localclaim}
\label{claim:En-computable}
$E_n: W_{n+1} \to \mathcal{A}_-(B_n)$ is computable for all $n\in\IN$.
\end{localclaim}

\noindent
{\bf Computable embedding of blocks into snakes.}
As stated above, the $S_i$ on different levels are connected via computable embeddings $f_n : B_{n+1} \hookrightarrow S_n$. The precise form of the $f_n$ is irrelevant for our purposes, we merely demand that they map every \emph{stripe} $C_{n+1}^k := [k,k+1] \times [0,h_{n+1}]$ for $k \leq w_n - 1$ in $B_{n+1}$ to a specific tile in $S_n$. Clearly, adjacent stripes have to be mapped into edge-adjacent tiles for a continuous embedding to exist, and this requirement is sufficient.
We will state our specific requirements inductively.

\begin{localclaim}
\label{claim:dim2-4}
There exists a computable sequence $(f_n)_n$ of computable embeddings $f_n:B_{n+1}\into S_n$ such that
for $n>0$, $i < w_n$ and $k \leq 2n-1$
\begin{enumerate}
\item $f_0(C^0_1) \subseteq [0,1] \times [0,1]$ \hfill (base cases)
\item $f_0(C^1_1) \subseteq [0,1] \times [1,2]$
\item $f_0(C^2_1) \subseteq [0,1] \times [2,3]$
\item $f_{n}(C_{n+1}^{k(w_n +1)+i}) \subseteq kc_{n} + f_{n-1}(C^i_n)$ \hfill (snake)
\item $f_n(C_{n+1}^{k(w_n +1)-1}) \subseteq kc_n +( [-1,0] \times [0,1])$ if $k\geq 1$ \hfill (padding stripe)
\item %$f_n(C_{n+1}^{2n(w_n+1) -1})=
         $f_n(C_{n+1}^{w_{n+1}-2w_n-2}) \subseteq [w_n-1,w_n] \times [h_n - 4, h_n -3]$ \hfill (padding stripe)
\item %$f_n(C_{n+1}^{2n(w_n+1) + i}) = 
        $f_n(C_{n+1}^{w_{n+1}-2w_n-1+i}) \subseteq [w_n-i-1,w_n-i] \times [h_n-3,h_n-2]$ \hfill (snake)
\item %$f_n(C_{n+1}^{(2n+1)(w_n + 1) - 1})=
        $f_n(C_{n+1}^{w_{n+1}-w_n-1}) \subseteq [0,1] \times [h_n-2,h_n-1]$ \hfill (padding stripe)
\item %$f_n(C_{n+1}^{(2n+1)(w_n+1) + i})=
        $f_n(C_{n+1}^{w_{n+1}-w_n+i}) \subseteq [i,i+1] \times [h_n-1, h_n]$ \hfill (snake) 
\end{enumerate}
\end{localclaim}
\begin{proof}
Clearly, adjacent stripes have to be mapped into edge-adjacent tiles for a continuous embedding $f_n:B_{n+1}\into S_n$ 
to exist, and this requirement is sufficient. According to the definition above this requirement is satisfied and since
each snake $S_n$ consists of finitely many tiles, there is also a computable embedding $f_n$ that satisfies these requirements.
Since the inductive definition of the sequence $(S_n)_n$ is computable in a uniform way depending on $n$, it follows
that there is also a computable sequence $(f_n)_n$ of embeddings that satisfies the given requirements.
\end{proof}

The embeddings $f_0:B_1\into S_0$ and $f_1:B_2\into S_1$ are illustrated in Figures~\ref{fig:f0} and \ref{fig:f1}.
The case of $f_1$ is already a prototype for the general situation of embeddings $f_n:B_{n+1}\into S_n$ with $n\geq1$.
Essentially, consecutive stripes of $B_{n+1}$ are mapped by $f_n$ into consecutive tiles in $S_n$ one by one.

The definition of the embeddings $f_n$ matches the definition of the sets $E_m$ in such a way that they are preserved
in a particular way. In order to express this result precisely, we first define a function $s_n:W_\infty\to W_{n +1}$ that removes
all irrelevant information from the input sequence or word $p$, i.e., it removes all entries with a first
component $k>n$.

\begin{localdef}
For all $n\in\IN$ we define a function $s_n:W_\infty\to W_{n+1}$ by:
\begin{enumerate}
\item $s_n(p):=\varepsilon$ if $k>n$ for all $k_b\in\range(p)$
\item $s_n(k_bu):=k_bs_n(u)$ if $k\leq n$ and
\item $s_n(k_bu):=s_n(u)$ if $k>n$.
\end{enumerate}
\end{localdef}

Now we can derive the following from the respective definitions. 

\begin{localclaim}
\label{claim:dim2-8}
$E_{m+1}(w)\In f_m^{-1}(E_m(s_m(w)))$ for all $w\in W_{m+2}$ and $m\in\IN$.
\end{localclaim}

\noindent
{\bf The reduction function.}
Now we can define the actual reduction function that we are going to use for the reduction
$\LLPO_\infty\leqSW\ConC_{B_0}$.

\begin{localdef}
We define a function 
\[A:W_\infty\to\AA_-(B_0),p\mapsto\bigcap_{n=0}^\infty (f_0\circ...\circ f_{n-1})(E_n(s_n(p))).\]
\end{localdef}

\begin{localclaim}
\label{claim:dim2-5}
$A:W_\infty\to\AA_-(B_0)$ is computable and given a point $x\in A(p)$, we can computably
reconstruct some $q\in\LLPO_\infty(p)$.
\end{localclaim}
\begin{proof}
Since
$(s_n)_n$ and $(E_n)_n$ are computable sequences, and since $E_n(s_n(p)))\In B_n$ and $(B_n)_n$ is a computable
sequence of compact sets, it follows that also $((f_0\circ...\circ f_{n-1})(E_n(s_n(p))))_n$ is a computable sequence
of compact sets whose intersection is compact and can be computed as a closed set.
Given a point $x\in A(p)$ we can reconstruct $\LLPO_\infty(p)(n)$ by computing 
$f_{n-1}^{-1}\circ...\circ f_0^{-1}(x)\in E_n(s_n(p))$. By Claim~\ref{claim:dim2-3} we have that
$n_0\in\range(p)$ implies $E_n(s_n(p))\In R_n$ and $n_1\in\range(p)$ implies $E_n(s_n(p))\In G_n$.
Since $R_n$ and $G_n$ consists of finitely many tiles and are clearly disjoint, we can find one possible
value for $\LLPO_\infty(p)(n)=b$ such that $n_b\not\in\range(p)$.
\end{proof}

What remains to be proved in order to conclude that $A$ yields the reduction $\LLPO_\infty\leqSW\ConC_{B_0}$ is
to show that $A(p)$ is always non-empty and connected.

\begin{localclaim}
\label{claim:dim2-6}
If the sets
\[A_m(w):=\bigcap_{n=0}^m(f_n\circ...\circ f_{m-1})^{-1}(E_n(s_n(w)))\In B_m\]
are non-empty and connected for all $w\in W_*$, then so is $A(p)$ for all $p\in W_\infty$.
\end{localclaim}
\begin{proof}
If $A_m(w)$ is non-empty and connected for all $w\in W_*$, then so is
\[K_m:=f_0\circ...\circ f_{m-1}(A_m(s_m(p)))=\bigcap_{n=0}^m f_0\circ...\circ f_{n-1}(E_n(s_n(p)))\In B_0\]
for all $p\in W_\infty$. Since the $K_m$ are also compact, $A(p)=\bigcap_{m=0}^\infty K_m$
is a decreasing chain of non-empty continua and hence itself a non-empty
continuum by \cite[Corollary~6.1.19]{Eng89}.
\end{proof}

Note that any $A_m(w)$ is a union of tiles in $S_m$. We can be more specific, though, and this will be useful in the proof.\\

\noindent
{\bf Connectedness of the sets $A(p)$.}
By mutual induction we define the notions of a segment in $S_n$ and a slice in $B_n$.
This will help us to prove that the sets $A(p)$ are connected.

\begin{localdef}[Segments and slices]
\label{def:segmentslice}
Let the notions of a \emph{segment} in $S_n$ and of a \emph{slice} in $B_n$ be defined by mutual induction for all $n\in\IN$:
\begin{enumerate}
\item\label{seg-slice1} $B_0$ is a slice in $B_0$.
\item\label{seg-slice2} If $G$ is a segment in $S_{n}$, then $f_{n}^{-1}(G)$ is a slice in $B_{n+1}$.
\item\label{seg-slice3} $S_n$ is a segment in $S_n$.
\item\label{seg-slice4} If $L$ is a slice in $B_n$, then $L\cap([0,w_n] \times [h_n-3, h_n-2])$ is a segment in $S_n$.
\item\label{seg-slice5} If $L$ is a slice in $B_n$, then $L\cap([0,w_n] \times [h_n-1, h_n])$ is a segment in $S_n$.
\item\label{seg-slice6} If $G$ is a segment in $S_{n-1}$ and $k \leq 2n-1$, then $kc_n + G$ is a segment in $S_n$.
\end{enumerate}
\end{localdef}

\begin{localclaim}
\label{claim:dim2-7}
We obtain the following for all $n\in\IN$:
\begin{enumerate}
\item $B_n$ is a slice in $B_n$.
\item Every slice in $B_n$ is of the form $[a,b] \times [0,h_n]$ with $a < b$.
\item Every segment in $S_n$ is non-empty and connected.
\item $E_n(w)$ is a segment in $S_n$ for every $w\in W_{n+1}$.
\end{enumerate}
\end{localclaim}
\begin{proof}
(1) follows from Definition~\ref{def:segmentslice} (\ref{seg-slice1}), (\ref{seg-slice2}) and (\ref{seg-slice3}).

(2) and (3) By mutual induction on the definition of segments and slices. 

(4) That $E_n(\varepsilon)=S_n$ is a segment follows from Definition \ref{def:segmentslice} (\ref{seg-slice3}). That $E_n(n_bw)$ is a segment follows from Definition \ref{def:segmentslice} (4,5) with the help of (1). The case $E_n(k_bw)$ with $k < n$ follows from Definition \ref{def:segmentslice} (6) by induction.
\end{proof}

Next we show that the sets $A_m$ and $E_m$ are essentially identical.

\begin{localclaim}
\label{claim:dim2-9}
$A_m=E_m\circ s_m$ for all $m\in\IN$.
\end{localclaim}
\begin{proof}
For $m=0$ the claim follows from the definition.
An inspection of the definition of $A_m(w)$ from Claim~\ref{claim:dim2-6} shows that for $m>0$
\begin{eqnarray}
\label{equ:dim2-9}
A_{m}(w)=f_{m-1}^{-1}(A_{m-1}(w))\cap E_{m}(s_{m}(w))
\end{eqnarray}
for all $w\in W_*$. We prove by induction on $m\in\IN$ that $A_m(w)=E_{m}(s_{m}(w))$ for all $w\in W_*$.
For $m=0$ this holds by definition of $A_0$. Given the induction claim for $m$, we obtain by Claim~\ref{claim:dim2-8}
\[E_{m+1}(w)\In f_m^{-1}(E_m(s_m(w)))\In f_m^{-1}(A_m(w))\] 
for all $w\in W_{m+2}$.
This in turn implies $A_{m+1}(w)=E_{m+1}(s_{m+1}(w))$ for all $w\in W_*$ using equation~(\ref{equ:dim2-9}).
\end{proof}

This easily implies the final claim.

\begin{localclaim}
$A(p)$ is non-empty and connected for all $p\in W_\infty$.
\end{localclaim}
\begin{proof}
By Claim~\ref{claim:dim2-6} it is sufficient to show that all the sets $A_m(w)$ for $w\in W_*$ are connected and non-empty.
However, this is the case according to Claims~\ref{claim:dim2-9} and \ref{claim:dim2-7}~(3,4).
\end{proof}
 
Together with Claim~\ref{claim:dim2-5} this completes the proof of $\LLPO_\infty\leqSW\ConC_{B_0}$ and
hence it completes the proof of Theorem~\ref{thm:dim2}.
\end{proof}

Even though Theorem~\ref{thm:dim2} completes our characterization of the Brouwer Fixed Point Theorem in dimension $2$,
it raises some further questions. The construction provided by Proposition~\ref{prop:twisted-cube} has the property that the values of
$T$ are even pathwise connected sets. Let us denote by $\PWCC_n$ the restriction of $\ConC_n$ to pathwise connected sets. Then a part of
Theorem~\ref{thm:dimension-three} can be strengthened to the following result.

\begin{corollary}[Pathwise connected choice]
$\PWCC_n\equivSW\C_{[0,1]}$ for all $n\geq 3$.
\end{corollary}

However, we are left with the following open question.

\begin{question}
Is $\PWCC_2\equivW\C_{[0,1]}$?
\end{question}

At least the construction in the proof of Theorem~\ref{thm:dim2} does not answer this question
since it yields a connected set $A(p)$ that is not pathwise connected.

We can draw some further interesting conclusions from the construction of the sets $A(p)$
that is related to the work of Iljazovi{\'c} \cite{Ilj09}, who studied computability properties
of chainable decomposable continua\footnote{Thanks to an anonymous referee for pointing
out this connection.}.
We recall that a continuum $A\In[0,1]^n$ is called {\em decomposable} if it is the union of two of its proper
subcontinua. And $A$ is called {\em chainable}, if for every $\varepsilon>0$ there exists 
an $\varepsilon$--chain $C_0,...,C_m$ that covers $A$. For $C_0,...,C_m\In[0,1]^n$ to be
an {\em $\varepsilon$--chain} means that the $C_0,...,C_m$ are non-empty open sets with $\diam(C_i)<\varepsilon$
and such that $C_i\cap C_j\not=\emptyset$ holds if and only if $|i-j|\leq1$.
The following is a consequence of \cite[Theorem~44]{Ilj09}.

\begin{proposition}[Iljazovi{\'c} 2009]
\label{prop:Ilja09}
Every co-c.e.\ chainable and decomposable continuum $A\In[0,1]^n$ contains a dense subset of computable points.
\end{proposition}

The construction of the sets $A(p)$ in the proof of Theorem~\ref{thm:dim2} guarantees that
there is a computable point $p$ such that $A(p)$ does not contain any computable point.
It is also easy to see that the sets $A(p)$ are chainable. As a conclusion we obtain the following.

\begin{corollary}
There is a non-empty co-c.e.\ chainable continuum $A\In[0,1]^2$ that does not contain any
computable point.
\end{corollary}

As a consequence of this results and Proposition~\ref{prop:Ilja09} it follows that the corresponding
set $A$ is necessarily indecomposable.

\section{The Displacement Principle}
\label{sec:displacement}

In this section, we want to prove a displacement principle that provides some information on the
power of binary choice $\C_{\{0,1\}}$ on the left-hand side of a reduction. 
We will apply this principle in Section~\ref{sec:idempotency} to prove that $\ConC_1$ is not idempotent.
In order to 
prove our result we first need to study the convergence relation of $\AA_-(X)$ induced by $\psi_-$.
This convergence relation can be characterized in terms of {\em closed upper limits} as defined by Hausdorff. 
For a sequence $(A_i)_{i\in\IN}$ of closed subsets of a topological space $X$ 
the {\em closed upper limit} of $(A_i)_{i\in\IN}$ is defined by
\[\Ls(A_i):=\bigcap_{k=0}^\infty\overline{\bigcup_{i=k}^\infty A_i}.\]
The common notation $\Ls$ is derived from the fact that this is also called the 
{\em topological limit superior} of $(A_i)_{i\in\IN}$. The set $\Ls(A_i)$ is always closed and possibly empty. 
If $X$ is compact and all the $A_i$ are non-empty, then $\Ls(A_i)$ is also compact and non-empty
by Cantor's Intersection Theorem.
We mention the following known characterization of the topological limit superior by Choquet 
(see, for instance, Proposition~5.2.2 in \cite{Bee93}).

\begin{fact}[Choquet]
\label{fact:Ls}
Let $X$ be a Hausdorff space and let $\NN_x$ denote the set of open neighborhoods of $x\in X$. For each
sequence $(A_i)_{i\in\IN}$ of closed sets $A_i\In X$ one has
\[\Ls(A_i)=\{x\in X:(\forall U\in\NN_x)(\forall k)(\exists i\geq k)\;U\cap A_i\not=\emptyset\}.\] 
\end{fact}

It is well-known that the topological limit superior (and the related topological limit inferior) are
used to define Kuratowski-Painlev\'e convergence, which is closely related to convergence with respect to the Fell topology 
(see Chapter~5 in \cite{Bee93}). 
Here we characterize the convergence relation of $\AA_-(X)$ in terms of the topological limit superior.
For a sequence $(A_i)_{i\in\IN}$ and a set $A$ in $\AA_-(X)$ we write $A_i\to A$ if there are $p_i$ and $p$
such that $\psi_-(p_i)=A_i$, $\psi_-(p)=A$ and $p_i\to p$. We note that this convergence relation on $\AA_-(X)$
is not unique in general, i.e., one sequence $(A_i)_{i\in\IN}$ can have many different limits. The following result
gives an exact characterization.

\begin{lemma}[Closed upper limit]
\label{lem:closed-upper-limit}
Let $X$ be a computable metric space and let $A_i,A\in\AA_-(X)$ for all $i\in\IN$.
Then $A_i\to A$ if and only if $\Ls(A_i)\In A$.
\end{lemma}
\begin{proof}
Let $p_i$ and $p$ be such that $\psi_-(p_i)=A_i$, $\psi_-(p)=A$.

We now assume $p_i\to p$.
Let $x\not\in A$. Then there is some basic open neighborhood $B_m$ of $x$ that
is eventually listed in position $j$ of $p$. Since the $p_i$ converge to $p$, there is
a $k\in\IN$ such that $B_m$ is also listed in position $j$ of $p_i$ for all $i\geq k$. 
According to Fact~\ref{fact:Ls} this means that $x\not\in\Ls(A_i)$.
Hence, we have proved $\Ls(A_i)\In A$.

Let us now assume that $\Ls(A_i)\In A$. It suffices to find $q_i$ with $\psi_-(q_i)=A_i$ and $q_i\to p$.
We choose $q_i:=p|_{m_i}p_i$, where $p|_{m_i}$ is the prefix of $p$ of suitable length $m_i$.
It is clear that $q_i\to p$ follows if the $m_i$ are increasing without bound, so we need to prove that we can choose 
such $m_i$ with $\psi_-(q_i)=\psi_-(p_i)$. We note that for each $n$ the set $U=B_{p(n)}$
does not intersect $A$, i.e., $U\cap A=\emptyset$ and hence there is some $k$ such that for all $i\geq k$
we have $U\cap A_i=\emptyset$ by Fact~\ref{fact:Ls}. That means that we can add the ball $B_{p(n)}$ to 
the negative information of $A_i$ without changing $A_i$. This guarantees the existence of a suitable
unbounded increasing sequence $m_i$.
\end{proof}

This result implies that the convergence relation on $\AA_-(X)$ induced by $\psi_-$ is the
convergence relation of the upper Fell topology and hence $\psi_-$ is admissible with respect to this
topology (which was already known, see \cite{Sch02c}).
We introduce some further terminology. 
If $\SS\In\AA_-(X)$, then we denote by 
\[\overline{\SS}:=\{A\in\AA_-(X):(\exists (A_i)_{i\in\IN}\in\SS^\IN)\;\Ls(A_i)\In A\}\] 
the {\em sequential closure} of $\SS$ in $\AA_-(X)$ and by 
\[2\SS:=\{A\in\SS:(\exists A_1,A_2\in\SS)(A_1\cap A_2=\emptyset\mbox{ and }A_1\cup A_2\In A)\}\]
the set of those sets in $\SS$ that have two disjoint subsets in $\SS$.
By $\C_X|_\SS$ we denote the restriction of $\C_X$ to $\SS$.

\begin{theorem}[Displacement Principle]
Let $f$ be a multi-valued function on represented spaces, let $X$ be a computable metric space
and let $\SS\In\AA_-(X)$. Then
\[f\times\C_{\{0,1\}}\leqW\C_X|_\SS\TO f\leqW\C_X|_{\SS\cap2\overline{\SS}}.\]
An analogous statement holds with $\leqW$ replaced by $\leqSW$ in both instances.
\end{theorem}
\begin{proof}
We use the computable metric space $(X,\delta_X)$ and represented spaces $(Y,\delta_Y)$ and $(Z,\delta_Z)$.
We assume that $f$ is of type $f:\In Y\mto Z$ and we use $\C_{\{0,1\}}\equivSW\LLPO$ (see \cite{BBP12}).
Let $H,K:\In\IN^\IN\to\IN^\IN$ be computable functions that witness the reduction $f\times\LLPO\leqW\C_X|_\SS$,
i.e., $H\langle\id,GK\rangle\vdash f\times\LLPO$ whenever $G\vdash\C_X|_\SS$.

We recall that $\LLPO:\In\IN^\IN\mto\IN$ is defined such that for $j\in\{0,1\}$ and $p\in\{0,1\}^\IN$ it holds that
$j\in\LLPO(p)\iff(\forall i)\;p(2i+j)=0$,
where $\dom(\LLPO)$ contains all sequences $p$ such that $p(k)\not=0$ for at most one $k$.
We consider the inputs $p_{j,i}:=0^{2i+j+1}10^\IN$ and $p_\infty:=0^\IN$ for $\LLPO$.
We obtain $\LLPO(p_{j,i})=\{j\}$ for $j\in\{0,1\}$ and $\LLPO(p_\infty)=\{0,1\}$.
For every $p\in\dom(f\delta_Y)$, $i\in\IN$ and $j\in\{0,1\}$ we now define
\[A_{j,i}^p:=\psi_-K\langle p,p_{j,i}\rangle\mbox{ and }A_\infty^p:=\psi_-K\langle p,p_\infty\rangle.\]
Since $p_{j,i}\to p_\infty$ for $i\to\infty$, continuity of $K$ implies $\Ls(A_{j,i}^p)\In A_\infty^p$ for $j\in\{0,1\}$
by Lemma~\ref{lem:closed-upper-limit}. Now we consider the corresponding subsets of $\dom(H)$:
\[B_j^p:=\bigcup_{i=0}^\infty\left\langle\{\langle p,p_{j,i}\rangle\}\times\delta_X^{-1}(A_{j,i}^p)\right\rangle,
  B_\infty^p:=\left\langle\{\langle p,p_\infty\rangle\}\times\delta_X^{-1}(A_\infty^p)\right\rangle.
\]
By $\pi_j$ we denote the projection on the $j$--th component of a tuple in Baire space.
Then $h:=\delta_\IN\pi_2 H:\In\IN^\IN\to\IN$ is a computable function such that 
$h|_{B_j^p}$ is constant with value $j$ for $j\in\{0,1\}$. 
We claim that due to continuity of $h$ this implies $\Ls(A_{0,i}^p)\cap\Ls(A_{1,i}^p)=\emptyset$.
Let us assume that $q$ is such that $\delta_X(q)\in\Ls(A_{0,i}^p)\cap\Ls(A_{1,i}^p)$.
Then $\delta_X(q)\in A_\infty^p$ and hence $r:=\langle\langle p,p_\infty\rangle,q\rangle\in B_\infty^p\In\dom(h)$.
Let now $U$ be a neighborhood of $r$ and let $j\in\{0,1\}$.
By Fact~\ref{fact:Ls} the point $\delta_X(q)$ is a cluster point of the sequence $(A_{j,i}^p)_{i\in\IN}$ and hence
there is a sequence $(q_i){i\in\IN}$ with $\delta_X(q_i)\in A_{j,i}^p$ for all $i$ with a subsequence that converges to $q$.
Hence, for some sufficiently large $i$ we obtain $\langle\langle p,p_{j,i}\rangle,q_i\rangle\in B_j^p\cap U$, 
which means $B_j^p\cap U\not=\emptyset$ for $j\in\{0,1\}$.
Hence $h|_U$ has to take both values $0$ and $1$ on any neighborhood $U$ of $r$, which contradicts continuity of $h$. 
This proves the claim $\Ls(A_{0,i}^p)\cap\Ls(A_{1,i}^p)=\emptyset$.

Altogether, we have proved $A_\infty^p\in2\overline{\SS}$ and $A_\infty^p\in\SS$ is clear.
We now define computable functions $H',K':\In\IN^\IN\to\IN^\IN$ by
\[H'\langle p,q\rangle:=\pi_1H\langle\langle p,p_\infty\rangle,q\rangle\mbox{ and }K'(p)=\pi_1K\langle p,p_\infty\rangle.\]
Then $H'\langle\id,GK'\rangle\vdash f$ whenever $G\vdash\C_X|_{\SS\cap2\overline{\SS}}$, i.e., $f\leqW\C_X|_{\SS\cap2\overline{\SS}}$.
If $H$ does not depend on the first component, then $H'=\pi_1H$ also does not depend on the first component. 
Hence the claim also holds for strong reducibility $\leqSW$ in place of $\leqW$. 
\end{proof}

If $\SS$ only contains non-empty closed sets $A\In X$ and $X$ is compact, then $\overline{\SS}$
also contains only non-empty sets and $2\overline{\SS}$ contains only sets that have at least
two points. Hence we obtain the following corollary, where $\UU(X):=\{\{x\}:x\in X\}$ denotes the
set of singleton subsets of $X$.

\begin{corollary}
\label{cor:displacement}
Let $f$ be a multi-valued function on represented spaces, 
let $X$ be a compact computable metric space and let $\SS\In\AA_-(X)\setminus\{\emptyset\}$.
Then
\[f\times\C_{\{0,1\}}\leqW\C_X|_\SS\TO f\leqW\C_X|_{\SS\setminus\UU(X)}.\]
An analogous statement holds with $\leqW$ replaced by $\leqSW$ in both instances.
\end{corollary}

\section{Idempotency of Connected Choice in Dimension One}
\label{sec:idempotency}

The goal of this section is to prove that connected choice $\ConC_1$ of dimension one is 
not idempotent, i.e., $\ConC_1\nequivW\ConC_1\times\ConC_1$.
For this purpose we use $\ConC_1^-$, which is just the restriction of $\ConC_1$ 
to such connected sets that are not singletons. In \cite{BG11a} it was proved that
$\ConC_1^-\leqW\C_\IN$, which follows since one can just guess a rational number in a 
non-degenerate interval and with finitely many mind changes one can find a correct number.
Using the Displacement Principle we can prove the following result.

\begin{proposition}
\label{prop:CC-factor}
$\ConC_1\lW\ConC_1\times\C_{\{0,1\}}$.
\end{proposition}
\begin{proof}
It is clear that $\ConC_1\leqW\ConC_1\times\C_{\{0,1\}}$. 
Let us assume that also $\ConC_1\times\C_{\{0,1\}}\leqW\ConC_1$.
Then by Corollary~\ref{cor:displacement} we obtain $\ConC_1\leqW\ConC_1^-$.
Since $\ConC_1^-\leqW\C_\IN$ by \cite[Proposition~3.8]{BG11a}), we obtain
$\ConC_1\leqW\C_\IN$, which is a contradiction to \cite[Lemma~4.9]{BG11a}.
\end{proof}

While this result shows that binary choice $\C_{\{0,1\}}$ enhances the power
of connected choice $\ConC_1$ if multiplied with it, products of binary choice
with itself are not that powerful, as the next result shows.

\begin{proposition}
\label{prop:binary-choice}
$\C_{\{0,1\}}^*\leqSW\ConC_1^-$.
\end{proposition}
\begin{proof}
Given a pair $\langle n,p\rangle$ as input to $\C_{\{0,1\}}^*$ we need
to construct a non-degenerate connected closed set $A\In[0,1]$ any point of which
allows us to find a point in $\C_{\{0,1\}}^n$ for input $p$. The input $p$ 
describes a product $A_1\times...\times A_n$ of non-empty sets $A_i\In\{0,1\}$
by an enumeration of the complement.

In order to construct $A$ we use an auxiliary tree of rational complexes with branching degree $2n$ 
in which each complex exists exactly of one rational interval $[a,b]$ with $a<b$. More precisely,
we start with the root $[\frac{1}{2n+2},\frac{1}{2n+1}]$ and on each successor node
of the tree we use $2n$ canonical pairwise disjoint subintervals of the previous interval,
sorted in the natural order.

We now describe how we use this tree to construct $A$. Given $p$ we start to produce the root interval
$[\frac{1}{2n+2},\frac{1}{2n+1}]$ as long as no negative information on any of the sets $A_1,...,A_n$
is available. If $A_k$ is the first of these sets that is determined by $p$,
then we proceed with child node number $2k-1$ or $2k$ depending on whether $A_k=\{0\}$ or $A_k=\{1\}$.
We then produce a description of the interval associated with this child node until further information on 
one of the sets $A_{k+1},...,A_n$ becomes available, in which case we proceed inductively as described above.

Altogether, this procedure produces an interval $I$ that is somewhere between the root level (in case that all the
sets $A_i$ remain undetermined) and level $n$ below the root level of the tree (in case that all the sets
$A_i$ are eventually determined). Given a point $x\in I$, we can find one of the (at most two) intervals $J$ on level $n$
that are closest to $x$ and included in $I$.  
Given $J$, we can reconstruct all decisions in the above algorithm 
and in this way we can produce a point $(x_1,...,x_n)\in A_1\times...\times A_n$.
\end{proof}

We mention that one can use the level (as introduced by Hertling \cite{Her96}) to prove that
$\C_{\{0,1\}}^*\lSW\ConC_1^-$. One can show that $\ConC_1^-$ has no level, whereas
the level of $\C_{\{0,1\}}^*$ is $\omega_0$. Since the level is preserved downwards by reducibility,
it follows that the reduction must be strict.

We arrive at the main result of this section.

\begin{theorem}[Non-idempotency]
\label{thm:idempotency}
$\ConC_1\lW\ConC_1\times\ConC_1\lW\ConC_2$.
\end{theorem}
\begin{proof}
Firstly, it is clear that $\ConC_1\lW\ConC_1\times\C_{\{0,1\}}\leqW\ConC_1\times\ConC_1$ holds
by Propositions~\ref{prop:CC-factor} and \ref{prop:binary-choice}.  
Secondly, it is also clear that $\ConC_1\times\ConC_1\leqW\ConC_2$, since
the product map $(A,B)\mapsto A\times B$ is computable on closed sets and
the product of two connected sets is connected. 
Finally, $\ConC_1\times\ConC_1$ is non-uniformly computable,
whereas $\ConC_2$ is not by Proposition~\ref{prop:dimension-two} and hence $\ConC_1\times\ConC_1\lW\ConC_2$. 
\end{proof}

In particular, $\ConC_1$ is not idempotent and the same reasoning that was used in the proof 
shows that $\ConC_2\nleqW\ConC_1^*$ holds for the idempotent closure $\ConC_1^*$.
That means that not even an arbitrary finite number of copies of $\ConC_1$ in parallel is powerful
enough to compute connected choice in dimension two.
With Corollary~\ref{cor:IVT} we obtain the following conclusion of Theorem~\ref{thm:idempotency}.

\begin{corollary}
The Brouwer Fixed Point Theorem $\BFT_1$ of dimension one and 
the Intermediate Value Theorem $\IVT$ are both not idempotent.
\end{corollary}

This means that two realizations of the Intermediate Value Theorem in parallel are more
powerful than just one. 

A problem related to idempotency is whether $\ConC_n$ is a cylinder. Again it is clear that
$\ConC_n$ is a cylinder for $n\geq 2$, which follows from Theorems~\ref{thm:dimension-three}
and~\ref{thm:dim2} and the fact that $\C_{[0,1]}$ is a cylinder. 
We can use the techniques of this section to prove that $\ConC_1$ is not a cylinder.

\begin{theorem}[Cylinder] $\ConC_1$ is not a cylinder.
\end{theorem}
\begin{proof}
Let us assume that $\id\times\ConC_1\leqSW\ConC_1$. Then $\id\times\C_{\{0,1\}}\leqSW\ConC_1$ follows 
by Proposition~\ref{prop:binary-choice} and hence $\id\leqSW\ConC_1^-$ by Corollary~\ref{cor:displacement}.
Since $\ConC_1^-$ has a realizer that always selects a rational number, we obtain $\ConC_1^-\leqSW\ConC_1^-|^{\IQ\cap[0,1]}$,
where $\ConC_1^-|^{\IQ\cap[0,1]}$ denotes the restriction of $\ConC_1^-$ to $\IQ\cap[0,1]$ in the image.
By \cite[Proposition~13.2]{BGM12} this implies that $\range(\id)$ is countable, which is a contradiction!
\end{proof}

\section{Conclusions}
\label{sec:conclusions}

We have systematically studied the uniform computational content of the Brouwer Fixed Point Theorem
for any fixed dimension and we have obtained a systematic classification for all dimensions. 
A problem that we have left open is the status of pathwise connected choice of dimension two.
Besides solving this open problem, one can proceed into several different direction.
For one, one could study generalizations of the Brouwer Fixed Point Theorem, such as
the Schauder Fixed Point Theorem or the Kakutani Fixed Point Theorem.
On the other hand, one could study results that are based on the Brouwer Fixed Point Theorem,
such as equilibrium existence theorems in computable economics (see for instance \cite{RW99a}).
Nash equilibria existence theorems for bimatrix games have been studied in \cite{Pau10},
and they can be seen to be strictly simpler than
the general Brouwer Fixed Point Theorem (in fact they can be considered as linear version of it).

\bibliographystyle{plain}
\bibliography{C:/Users/vbrattka/Dropbox/Bibliography/lit}

\begin{thebibliography}{10}

\bibitem{Bai85}
G{\"u}nter Baigger.
\newblock Die {N}ichtkonstruktivit{\"a}t des {B}rouwerschen {F}ixpunktsatzes.
\newblock {\em Archive for Mathematical Logic}, 25:183--188, 1985.

\bibitem{Bee93}
Gerald Beer.
\newblock {\em Topologies on Closed and Closed Convex Sets}, volume 268 of {\em
  Mathematics and Its Applications}.
\newblock Kluwer Academic, Dordrecht, 1993.

\bibitem{Bra99}
Vasco Brattka.
\newblock Computable invariance.
\newblock {\em Theoretical Computer Science}, 210:3--20, 1999.

\bibitem{Bra05}
Vasco Brattka.
\newblock Effective {B}orel measurability and reducibility of functions.
\newblock {\em Mathematical Logic Quarterly}, 51(1):19--44, 2005.

\bibitem{Bra08}
Vasco Brattka.
\newblock Plottable real number functions and the computable graph theorem.
\newblock {\em SIAM Journal on Computing}, 38(1):303--328, 2008.

\bibitem{BBP12}
Vasco Brattka, Matthew de~Brecht, and Arno Pauly.
\newblock Closed choice and a uniform low basis theorem.
\newblock {\em Annals of Pure and Applied Logic}, 163:986--1008, 2012.

\bibitem{BG09}
Vasco Brattka and Guido Gherardi.
\newblock Borel complexity of topological operations on computable metric
  spaces.
\newblock {\em Journal of Logic and Computation}, 19(1):45--76, 2009.

\bibitem{BG11a}
Vasco Brattka and Guido Gherardi.
\newblock Effective choice and boundedness principles in computable analysis.
\newblock {\em The Bulletin of Symbolic Logic}, 17(1):73--117, 2011.

\bibitem{BG11}
Vasco Brattka and Guido Gherardi.
\newblock Weihrauch degrees, omniscience principles and weak computability.
\newblock {\em The Journal of Symbolic Logic}, 76(1):143--176, 2011.

\bibitem{BGHP17}
Vasco Brattka, Guido Gherardi, Rupert H\"olzl, and Arno Pauly.
\newblock The {V}itali covering theorem in the {W}eihrauch lattice.
\newblock In Adam Day, Michael Fellows, Noam Greenberg, Bakhadyr Khoussainov,
  Alexander Melnikov, and Frances Rosamond, editors, {\em Computability and
  Complexity: Essays Dedicated to Rodney G. Downey on the Occasion of His 60th
  Birthday}, volume 10010 of {\em Lecture Notes in Computer Science}, pages
  188--200. Springer, Cham, 2017.

\bibitem{BGM12}
Vasco Brattka, Guido Gherardi, and Alberto Marcone.
\newblock The {B}olzano-{W}eierstrass theorem is the jump of weak
  {K}{\H{o}}nig's lemma.
\newblock {\em Annals of Pure and Applied Logic}, 163:623--655, 2012.

\bibitem{BGP17}
Vasco Brattka, Guido Gherardi, and Arno Pauly.
\newblock Weihrauch complexity in computable analysis.
\newblock arXiv 1707.03202, 2017.

\bibitem{BHK17a}
Vasco Brattka, Matthew Hendtlass, and Alexander~P. Kreuzer.
\newblock On the uniform computational content of computability theory.
\newblock {\em Theory of Computing Systems}, 61(4):1376--1426, 2017.

\bibitem{BHK18}
Vasco Brattka, Matthew Hendtlass, and Alexander~P. Kreuzer.
\newblock On the uniform computational content of the {B}aire category theorem.
\newblock {\em Notre Dame Journal of Formal Logic}, 59(4):605--636, 2018.

\bibitem{BLRMP16}
Vasco Brattka, St{\'e}phane Le~Roux, Joseph~S. Miller, and Arno Pauly.
\newblock The {B}rouwer fixed point theorem revisited.
\newblock In Arnold Beckmann, Laurent Bienvenu, and Nata\v{s}a Jonoska,
  editors, {\em Pursuit of the Universal}, volume 9709 of {\em Lecture Notes in
  Computer Science}, pages 58--67, Switzerland, 2016. Springer.
\newblock 12th Conference on Computability in Europe, CiE 2016, Paris, France,
  June 27 - July 1, 2016.

\bibitem{BLP12a}
Vasco Brattka, St{\'e}phane Le~Roux, and Arno Pauly.
\newblock On the computational content of the {B}rouwer {F}ixed {P}oint
  {T}heorem.
\newblock In S.~Barry Cooper, Anuj Dawar, and Benedikt L{\"o}we, editors, {\em
  How the World Computes}, volume 7318 of {\em Lecture Notes in Computer
  Science}, pages 57--67, Berlin, 2012. Springer.
\newblock Turing Centenary Conference and 8th Conference on Computability in
  Europe, CiE 2012, Cambridge, UK, June 2012.

\bibitem{BP03}
Vasco Brattka and Gero Presser.
\newblock Computability on subsets of metric spaces.
\newblock {\em Theoretical Computer Science}, 305:43--76, 2003.

\bibitem{BW99}
Vasco Brattka and Klaus Weihrauch.
\newblock Computability on subsets of {E}uclidean space {I}: Closed and compact
  subsets.
\newblock {\em Theoretical Computer Science}, 219:65--93, 1999.

\bibitem{Bro11}
L.~E.~J. Brouwer.
\newblock \"{U}ber {A}bbildung von {M}annigfaltigkeiten.
\newblock {\em Mathematische Annalen}, 71(1):97--115, 1911.

\bibitem{Bro52}
L.E.J. Brouwer.
\newblock An intuitionist correction of the fixed-point theorem on the sphere.
\newblock {\em Proceedings of the Royal Society. London. Series A.}, 213:1--2,
  1952.

\bibitem{CDJS93}
Douglas Cenzer, Rodney Downey, Carl Jockusch, and Richard~A. Shore.
\newblock Countable thin {$\Pi^0_1$} classes.
\newblock {\em Annals of Pure and Applied Logic}, 59(2):79--139, 1993.

\bibitem{CR98}
Douglas Cenzer and J.B. Remmel.
\newblock {$\Pi^0_1$}--classes in mathematics.
\newblock In Yu.~L. Ershov, S.S. Goncharov, A.~Nerode, J.B. Remmel, and V.W.
  Marek, editors, {\em Handbook of Recursive Mathematics}, volume 139 of {\em
  Studies in Logic and the Foundations of Mathematics}, pages 623--821.
  Elsevier, Amsterdam, 1998.
\newblock Volume 2, Recursive Algebra, Analysis and Combinatorics.

\bibitem{Col08}
Pieter Collins.
\newblock Computability and representations of the zero set.
\newblock In Vasco Brattka, Ruth Dillhage, Tanja Grubba, and Angela Klutsch,
  editors, {\em {CCA} 2008, Fifth International Conference on Computability and
  Complexity in Analysis}, volume 221 of {\em Electronic Notes in Theoretical
  Computer Science}, pages 37--43. Elsevier, 2008.
\newblock CCA 2008, Fifth International Conference, Hagen, Germany, August
  21--24, 2008.

\bibitem{Col09}
Pieter Collins.
\newblock Computability of homology for compact absolute neighbourhood
  retracts.
\newblock In Andrej Bauer, Peter Hertling, and Ker-I Ko, editors, {\em {CCA}
  2009, Proceedings of the Sixth International Conference on Computability and
  Complexity in Analysis}, pages 107--118, Schloss Dagstuhl, Germany, 2009.
  Leibniz-Zentrum f\"ur Informatik.

\bibitem{Eng89}
Ryszard Engelking.
\newblock {\em General Topology}, volume~6 of {\em Sigma series in pure
  mathematics}.
\newblock Heldermann, Berlin, 1989.

\bibitem{GM09}
Guido Gherardi and Alberto Marcone.
\newblock How incomputable is the separable {H}ahn-{B}anach theorem?
\newblock {\em Notre Dame Journal of Formal Logic}, 50(4):393--425, 2009.

\bibitem{Her96}
Peter Hertling.
\newblock {U}nstetigkeitsgrade von {F}unktionen in der effektiven {A}nalysis.
\newblock Informatik Berichte 208, FernUniversit\"at Hagen, Hagen, November
  1996.
\newblock Dissertation.

\bibitem{Her99b}
Peter Hertling.
\newblock An effective {R}iemann {M}apping {T}heorem.
\newblock {\em Theoretical Computer Science}, 219:225--265, 1999.

\bibitem{HRW12}
Mathieu Hoyrup, Crist{\'o}bal Rojas, and Klaus Weihrauch.
\newblock Computability of the {R}adon-{N}ikodym derivative.
\newblock {\em Computability}, 1(1):3--13, 2012.

\bibitem{Ilj09}
Zvonko Iljazovi{\'c}.
\newblock Chainable and circularly chainable co-r.e. sets in computable metric
  spaces.
\newblock {\em Journal of Universal Computer Science}, 15(6):1206--1235, 2009.

\bibitem{Ish06}
Hajime Ishihara.
\newblock Reverse mathematics in {B}ishop's constructive mathematics.
\newblock {\em Philosophia Scientiae, Cahier special}, 6:43--59, 2006.

\bibitem{KMM04}
Tomasz Kaczynski, Konstantin Mischaikow, and Marian Mrozek.
\newblock {\em Computational homology}, volume 157 of {\em Applied Mathematical
  Sciences}.
\newblock Springer-Verlag, New York, 2004.

\bibitem{Kih12}
Takayuki Kihara.
\newblock Incomputability of simply connected planar continua.
\newblock {\em Computability}, 1(2):131--152, 2012.

\bibitem{Koh05b}
Ulrich Kohlenbach.
\newblock Higher order reverse mathematics.
\newblock In Stephen~G. Simpson, editor, {\em Reverse Mathematics 2001},
  volume~21 of {\em Lecture Notes in Logic}, pages 281--295, Wellesley, 2005. A
  K Peters.

\bibitem{LRP15a}
St{\'{e}}phane Le~Roux and Arno Pauly.
\newblock Finite choice, convex choice and finding roots.
\newblock {\em Logical Methods in Computer Science}, 11(4):4:6, 31, 2015.

\bibitem{LZ08a}
St\'ephane Le~Roux and Martin Ziegler.
\newblock Singular coverings and non-uniform notions of closed set
  computability.
\newblock {\em Mathematical Logic Quarterly}, 54(5):545--560, 2008.

\bibitem{Mil02a}
Joseph~Stephen Miller.
\newblock {\em Pi-0-1 Classes in Computable Analysis and Topology}.
\newblock PhD thesis, Cornell University, Ithaca, USA, 2002.

\bibitem{Neu15}
Eike Neumann.
\newblock Computational problems in metric fixed point theory and their
  {W}eihrauch degrees.
\newblock {\em Logical Methods in Computer Science}, 11:4:20,44, 2015.

\bibitem{Ore63}
Vladimir~P. Orevkov.
\newblock A constructive mapping of the square onto itself displacing every
  constructive point ({R}ussian).
\newblock {\em Doklady Akademii Nauk}, 152:55--58, 1963.
\newblock translated in: Soviet Math. - Dokl., 4 (1963) 1253--1256.

\bibitem{Pau09}
Arno Pauly.
\newblock How discontinuous is computing {N}ash equilibria? ({E}xtended
  abstract).
\newblock In Andrej Bauer, Peter Hertling, and Ker-I Ko, editors, {\em 6th
  International Conference on Computability and Complexity in Analysis
  (CCA'09)}, volume~11 of {\em OpenAccess Series in Informatics (OASIcs)},
  Dagstuhl, Germany, 2009. Schloss Dagstuhl--Leibniz-Zentrum fuer Informatik.

\bibitem{Pau10}
Arno Pauly.
\newblock How incomputable is finding {N}ash equilibria?
\newblock {\em Journal of Universal Computer Science}, 16(18):2686--2710, 2010.

\bibitem{Pot08}
Petrus~H. Potgieter.
\newblock Computable counter-examples to the {B}rouwer fixed-point theorem.
\newblock arXiv 0804.3199, 2008.

\bibitem{RW99a}
Marcel~K. Richter and Kam-Chau Wong.
\newblock Non-computability of competitive equilibrium.
\newblock {\em Economic Theory}, 14(1):1--27, 1999.

\bibitem{Sch02c}
Matthias Schr{\"o}der.
\newblock {\em Admissible Representations for Continuous Computations}.
\newblock PhD thesis, Fachbereich Informatik, FernUniversit\"{a}t Hagen, 2002.

\bibitem{ST90}
Naoki Shioji and Kazuyuki Tanaka.
\newblock Fixed point theory in weak second-order arithmetic.
\newblock {\em Annals of Pure and Applied Logic}, 47:167--188, 1990.

\bibitem{Sim99}
Stephen~G. Simpson.
\newblock {\em Subsystems of Second Order Arithmetic}.
\newblock Perspectives in Mathematical Logic. Springer, Berlin, 1999.

\bibitem{Soa87}
Robert~I. Soare.
\newblock {\em Recursively Enumerable Sets and Degrees}.
\newblock Perspectives in Mathematical Logic. Springer, Berlin, 1987.

\bibitem{Ste89}
Thorsten~von Stein.
\newblock {\em Vergleich nicht konstruktiv l{\"o}sbarer {P}robleme in der
  {A}nalysis}.
\newblock PhD thesis, Fachbereich Informatik, FernUniversit\"at Hagen, 1989.
\newblock {D}iplomarbeit.

\bibitem{Wei92a}
Klaus Weihrauch.
\newblock The degrees of discontinuity of some translators between
  representations of the real numbers.
\newblock Technical Report TR-92-050, International Computer Science Institute,
  Berkeley, July 1992.

\bibitem{Wei92c}
Klaus Weihrauch.
\newblock The {TTE}-interpretation of three hierarchies of omniscience
  principles.
\newblock Informatik Berichte 130, FernUniversit\"at Hagen, Hagen, September
  1992.

\bibitem{Wei00}
Klaus Weihrauch.
\newblock {\em Computable Analysis}.
\newblock Springer, Berlin, 2000.

\end{thebibliography}
%\bibliography{C:/Users/vasco/Dropbox/Bibliography/lit}

\end{document}